\documentclass[11pt,leqno]{article}
\usepackage{amsthm,amsfonts,amssymb,amsmath,oldgerm}
\usepackage{epsfig}
\usepackage{graphicx}
\renewcommand\d{\partial}
\renewcommand\a{\alpha}
\renewcommand\b{\beta}

\newcommand\s{\sigma}

\newcommand\R{\mathbb R}\newcommand\N{\mathbb N}
\newcommand\C{\mathbb C}
\newcommand\un{\underline}
\def\g{\gamma}

\def\s{\sigma}

\def\k{\kappa}

\def\l{\lambda}
\def\eps{\varepsilon }
\def\e{\varepsilon}

\def\G{\Gamma}
\def\L{\Lambda}
\def\D{\partial}
\newcommand\kernel{\hbox{\rm Ker}}

\newcommand\br{\begin{rem}}
\newcommand\er{\end{rem}}
\newcommand\bp{\begin{pmatrix}}
\newcommand\ep{\end{pmatrix}}
\newcommand\be{\begin{equation}}
\newcommand\ee{\end{equation}}
\newcommand\ba{\begin{equation}\begin{aligned}}
\newcommand\ea{\end{aligned}\end{equation}}

\newcommand\cU{{\cal  U}}

\newcommand\cR{{\cal  R}}

\newcommand\cS{{\mathcal S}}

\newcommand{\RR}{{\mathbb R}}

\newcommand{\NN}{{\mathbb N}}

\newcommand{\const}{\text{\rm constant}}
\newcommand{\Id}{{\rm Id }}

\newtheorem{defi}{Definition}[section]
\newtheorem{theo}[defi]{Theorem}
\newtheorem{prop}[defi]{Proposition}
\newtheorem{lem}[defi]{Lemma}
\newtheorem{cor}[defi]{Corollary}
\newtheorem{rem}[defi]{Remark}
\newtheorem{ass}[defi]{Assumption}

\newtheorem{ex}[defi]{Example}

\def\op{{\rm op} }

\numberwithin{equation}{section}

\begin{document}

\renewcommand{\refname}{References}

\title{Nash--Moser iteration and singular perturbations}

\author{\sc \small Benjamin Texier\thanks{
Universit\'e Paris Diderot (Paris 7), Institut de Math\'ematiques de Jussieu, UMR CNRS 7586;
texier@math.jussieu.fr:
Research of B.T.  was partially supported
under NSF grant number DMS-0505780. Ce travail a b\'en\'efici\'e d'une aide de l'Agence Nationale de la Recherche portant la r\'ef\'erence ANR-08-BLAN-0301-01
},\,\, Kevin Zumbrun\thanks{Indiana University, Bloomington, IN 47405;
kzumbrun@indiana.edu:
Research of K.Z. was partially supported
under NSF grants no. DMS-0300487 and DMS-0801745.
 }}

\maketitle

\begin{abstract} 
We present a simple and easy-to-use Nash--Moser iteration theorem tailored
for singular perturbation problems admitting a formal asymptotic
expansion or other family of approximate solutions depending on
a parameter $\eps\to 0.$ The novel feature is
to allow loss of powers of $\eps$ as well as the usual
loss of derivatives in the solution operator for 
the associated linearized problem.
We indicate the utility of this theorem 
by describing sample applications to
(i) 
systems of quasilinear Schr\"odinger equations,
 and (ii) existence of
small-amplitude profiles of quasilinear
relaxation systems in the degenerate case that the velocity of the profile is a characteristic mode of the hyperbolic operator. 
\end{abstract} 

\tableofcontents
\begin{section}{Introduction}

Because the expansions themselves furnish arbitrarily accurate
approximate solutions,
and because the associated linearized estimates are often
stiff in terms of amplitude and or smoothness,
Nash--Moser iteration appears particularly well-adapted
to the verification of asymptotic expansions such as arise
in various singular perturbation problems depending on a 
small parameter $\eps\to 0$.
However, standard Nash--Moser theorems allow only for loss
of derivatives and not loss of powers of $\eps$ in the estimates
on the linearized solution operator, so that to apply Nash--Moser
iteration to problems that do lose powers of $\eps$
would appear to require a careful accounting of
constants throughout the entire Nash--Moser iteration
to check that the argument closes.

The purpose of this article therefore is to present 
a simple and general-purpose theorem carrying out this
accounting, which can be applied as an easy-to-use black box to this type
of problem.
We conclude by presenting two sample applications for which both
loss of derivatives and of powers of $\eps$ naturally occur for
the linearized problem, one
for systems of quasilinear Schr\"odinger equations
and one in  existence of
small-amplitude profiles of quasilinear relaxation systems.
The latter, due to M\'etivier, Texier and Zumbrun,
was treated in \cite{MTZ} by the approach presented here; though special cases may be treated by other methods
\cite{YZ,MaZ1,MaZ2,DY}, we do not know of any other solution in the 
generality considered there.

Our approach follows a very simple proof given by
Xavier Saint-Raymond \cite{XSR} of a (parameter-independent)
Nash--Moser implicit function theorem \cite{N,Mo} in a Sobolev space setting.
A novel aspect is our treatment of uniqueness, which we have not seen
elsewhere- in particular the incorporation of a phase condition in the case
that the linearized operator has a kernel. (See Theorem \ref{th2}.)

 We note that a parameter-dependent Nash--Moser scheme was recently used by Alvarez-Samaniego and Lannes \cite{AL} to prove local-in-time well-posed\-ness of model equations in oceanography.
Alvarez-Samaniego and Lannes 
do not allow losses in $\eps$ in the linearized solution operator,
which is the main point here. 

In \cite{IPT}, Iooss, Plotnikov and Tolland use a parameter-dependent Nash-Moser to prove the existence of periodic standing water-waves in the case of an infinite depth. They deal with a situation in which the linearized operator might not be invertible on small sets on $\e,$ and proceed by taking out corresponding bad regions in $\e.$ Here we consider everywhere invertible linearized operators, with losses (materialized by inverse powers of $\e$ in the estimates) for all values of the parameter $\e,$ and proceed by keeping track of these losses. 

An important reference on Nash--Moser-type theorems is 
Hamilton \cite{H}. 
Another good reference is Alinhac and G\'erard's book \cite{AG}.

\medbreak
{\bf Plan of the paper and scheme of the main proof.} 

 In Section \ref{sec:thm}, we state
the main theorem on $\eps$-dependent Nash--Moser iteration,
giving the proof afterward in Section \ref{sec:pf}. 

 The proof classically combines an iteration step (Newton's method) with a regularization procedure (in Sobolev spaces, high-frequency truncation). It is largely based on Xavier Saint-Raymond's elegant and short proof \cite{XSR}. The key in our context is to show that the bounds can be made uniform in the small parameter. Our main observation is that this can be achieved under the condition that the regularizing sequence is diverging to $+\infty$ fast enough, depending on the small parameter (see \eqref{def-theta}).

 In Section \ref{sec:app1}, we describe applications to 
systems of quasilinear Schr\"odinger equations, and %
in Section \ref{sec:app2} to
existence of small-amplitude profiles of 
quasilinear relaxation systems in the degenerate case that the velocity of the relaxation profile is a characteristic mode of the hyperbolic operator.

\end{section}

\section{A simple Nash--Moser theorem}\label{sec:thm}

\subsection{Setting}

 Consider two families of Banach spaces $\{E_s\}_{s \in \R},$ $\{F_s\}_{s \in \R},$ and a family of equations
 \begin{equation} \label{0}
  \Phi^\e(u^\e) = 0, \qquad u^\e \in E_s,
 \end{equation}
 indexed by $\e \in (0,1),$ where for some $m \geq 0,$ $s_0, s_1 \in \R,$ with $s_0 + 2 m \leq s_1,$ for all $\e,$
 \begin{equation} \label{0.1} \Phi^\e \in C^2(E_s,F_{s-m}), \quad \mbox{ for all $s_0 + m \leq s \leq s_1.$}\end{equation}
  Let $| \cdot |_s$ denote the norm
in $E_s$ and $\| \cdot \|_s$ denote the norm
in $F_s.$ The norms $|\cdot|_s$ and $\| \cdot\|_s$ 
and spaces $E_s$ and $F_s$ 
are possibly $\e$-dependent, as in our applications in Sections \ref{sec:app1} and \ref{sec:app2}. We assume that the embeddings
   \begin{equation} \label{sp0} E_{s'} \hookrightarrow E_{s}, \qquad F_{s'} \hookrightarrow F_{s}, \qquad s \leq s',\end{equation}
   hold, and have norms less than one:
   \begin{equation} \label{sp1} |\cdot|_s \leq |\cdot |_{s'}, \quad \| \cdot \|_s \leq \|\cdot \|_{s'}, \qquad s \leq s'.\end{equation}
  We assume the interpolation property:
   \begin{equation} \label{sp2}
    | \cdot|_{s + \s} \lesssim |\cdot|_{s}^{\frac{\s' - \s}{\s'}} |\cdot|_{s + \s'}^{\frac{\s}{\s'}}, \qquad 0 < \s < \s',
       \end{equation}
 where $| u |_s \lesssim | v|_{s'}$ stands for $| u|_s \leq C |v|_{s'},$ for some $C > 0$ possibly depending on $s$ and $s'$ but not on $\e,$ nor on $u$ and $v.$ We assume in addition the existence of a family of regularizing operators
  $$S_\theta: \quad E_s \to E_s, \qquad \theta  > 0,$$
   such that for all $s \leq s',$
 \begin{equation} \label{3}
   | S_\theta u - u |_{s} \lesssim \theta^{s - s'} | u |_{s'}.
 \end{equation}
\begin{equation} \label{4}
   | S_\theta u |_{s'} \lesssim (1 + \theta^{s' - s}) | u |_{s}.
 \end{equation}

 \begin{ex} \label{ex1} Let 
 $$E_s = H^s(\R^d), \qquad F_s = H^{s}(\R^d) \times H^{s+1}(\R^d),$$
 with $\e$-dependent norms defined by 
 \begin{equation} \label{hse} | v |_s := \| v \|_{H^s_\e} :=  \| (1 + | \e \xi|^2)^{s/2} ({\cal F} v)(\xi) \|_{L^2(\R^d_\xi)}, 
 \qquad \| (f,g) \|_{s} = |f|_{s} + |g|_{s+1},
 \end{equation}
 where ${\cal F} v$ denotes the Fourier transform of $v.$ 
 Then \eqref{sp0}, \eqref{sp1} and \eqref{sp2} hold.
A family of regularizing operators $E_s \to E_s$ is given by
 $$ S_\theta: \quad u \to S_\theta(u) := {\cal F}^{-1} \big( \chi(\theta^{-1} \xi) \hat u \big),$$
 where $\chi: \R^d \to [0,1]$ is a smooth truncation function, identically equal to $1$ for $|\xi| \leq 1,$ and identically equal to $0$ for $|\xi| \geq 2.$ 
 The family $\{S_\theta\}_{\theta > 0}$ satisfies \eqref{3} and \eqref{4}. 
 \end{ex}
 
  \begin{rem} \label{rem:Moser} The family of norms $\| \cdot \|_{H^s_\e}$ satisfies Moser's inequality
 $$ \| u v \|_{H^s_\e} \lesssim |u|_{L^\infty} \| v \|_{H^s_\e} + \| u \|_{H^s_\e} |v|_{L^\infty}, \qquad s \geq 0, \quad u, v \in H^s \cap L^\infty,$$
 and, if $F$ is smooth and satisfies $F(0) = 0,$ for some non-decreasing $C: \R_+ \to \R_+,$
 $$ \|F(u)\|_{H^s_\e} \leq C(|u|_{L^\infty}) \|u\|_{H^s_\e}, \qquad s > d/2, \quad u \in H^s.$$
 \end{rem}

  \begin{ex} \label{ex2} Consider the family of maps 
  $$\Phi^\e(u) = \big(\sum_{1 \leq j \leq d} A_j(u) \e \d_{x_j} u, u\, \big),$$
   where $A_j: u \in \R^n \to A_j(u) \in \R^{n \times n}$ is smooth. By Moser's inequality (Remark {\rm \ref{rem:Moser}}), for $s > 1 + d/2,$ the application $\Phi^\e$ maps smoothly $H^{s}$ to $H^{s-1} \times H^{s},$ that is, with the notation of Example {\rm \ref{ex1}}, $E_s$ to $F_{s-1}.$ 
   \end{ex}

 \subsection{First assumption: tame direct bounds}
 We assume bounds for $\Phi^\e$ and its first two derivatives.

      \begin{ass} \label{ass1} For some $\g_0 \geq 0,$ $\g_1 \geq 0,$  
 for all $s$ such that $s_0 + m \leq s \leq s_1 - m,$  for all $u,v,w \in E_{s+m},$ there holds
  \begin{equation} \label{1}
  \| \Phi^\e(u)\|_s \leq {\bf C}_0 (1 + | u |_{s+m})
  \end{equation} 
  \begin{equation}
  \| (\Phi^\e)'(u) \cdot v\|_s  \leq {\bf C}_1 ( \e^{-\g_1} |v|_{s_0 + m} |u|_{s +m} + |v|_{s+m}), \label{1'} 
  \end{equation}
  \begin{equation} \begin{aligned}
  \label{1''} \| (\Phi^\e)''(u) \cdot (v, w)\|_s  \leq {\bf C}_2 \Big( & \e^{-2 \g_1} |v|_{s_0 + m} |w|_{s_0 + m}|u|_{s + m} \\ & \quad + \e^{-\g_1} |w|_{s_0+m} |v|_{s + m} + \e^{-\g_1} | v|_{s_0+m} |w|_{s + m}\Big), \end{aligned}
  \end{equation}
    where the functions
    ${\bf C}_j = {\bf C}_j\big(\e, (|u|_{s_0 + \tilde m})_{0 \leq \tilde m \leq m}\big)$ satisfy
   $$ %
  \sup \Big\{ {\bf C}_j, \quad j = 0,1,2, \, \e \in (0,1), \,|u|_{s_{0} + m} \lesssim \e^{\gamma_0} \Big\}  < + \infty.
  $$ %
   \end{ass}
   
  The simplest example is given by a product map:
 
\begin{ex} \label{ex:product} Consider the map $\Phi_0: u \in H^s(\R^d;\R) \to (u^2,u) \in H^s \times H^s(\R^d;\R),$ for $d/2 < s.$ There holds, by Remark {\rm \ref{rem:Moser}},
$$ \|\Phi_0(u) \|_{H^s_\e \times H^s_\e} \lesssim |u|_{L^\infty} \|u\|_{H^s_\e} + \| u \|_{H^s_\e},$$
where the weighted norm $\| \cdot \|_{H^s_\e}$ is defined in \eqref{hse}.
 For small $\e,$ the classical Sobolev embedding $H^s \hookrightarrow L^\infty,$ for $d/2 < s_0 \leq s,$ has a large norm if $H^s$ is equipped with the weighted norm $\| \cdot \|_{H^s_\e}:$ 
  $$|u|_{L^\infty} \lesssim \e^{-d/2} \| u \|_{H^{s_0}_\e},$$
  so that 
$$ \| \Phi_0(u) \|_{H^s_\e \times H^s_\e} \lesssim (1 + \e^{-d/2} \| u \|_{H^{s_0}_\e}) \| u \|_{H^s_\e},$$
and
$$ \| \Phi'_0(u) \cdot v \|_{H^s_\e \times H^s_\e} \lesssim (1 + \e^{-d/2} \| u \|_{H^{s_0}_\e}) \| v \|_{H^s_\e} + \e^{-d/2} \| v \|_{H^{s_0}_\e} \| u \|_{H^s_\e}.$$
In particular, Assumption {\rm \ref{ass1}} holds with $m = 0,$ $\g_0 = \g_1 = d/2.$ 
\end{ex} 
   
   Another simple example is given by the map defined in Example \ref{ex2}:
   
 \begin{ex} \label{ex3} Consider the map $\Phi^\e: E_s \to F_{s-1},$ introduced in Example {\rm \ref{ex2}}, where the families of functional spaces $E_s$ and $F_{s}$ and their $\e$-dependent norms were introduced in Example {\rm \ref{ex1}}. Let $d/2 < s_0 < 1 + d/2$ be given. Just like in the above Remark, for small $\e,$ the classical Sobolev embedding 
  $$ 
  H^{N + s_0}(\R^d) \hookrightarrow W^{N,\infty}(\R^d), \qquad \quad N \in \N,
  $$ %
 has a large norm
 \begin{equation} \label{embed-norm}
  |v|_{W^{N,\infty}} \lesssim \e^{-N-d/2} |v|_{N + s_0},
  \end{equation}
 where $|\cdot |_s = \| \cdot \|_{H^s_\e}$ is defined in \eqref{hse}. By Remark {\rm \ref{rem:Moser}}, for $s \geq s_0,$
  $$ | A_j(u) \e \d_{x_j} u |_s \leq C_{A_j}(|u|_{L^\infty}) (|u|_{s+1} + |\e \d_{x_j} u|_{L^\infty} |u|_s ),$$
  where $C_{A_j}: \R_+ \to \R_+$ is non-decreasing and depends only on $A_j,$ $s$ and $d.$ By \eqref{embed-norm}, for $s \geq s_0 + 1,$
 $$  | A_j(u) \e \d_{x_j} u |_s \leq C_{A_j}\big( \e^{-d/2} | u |_{s_0}\big)(1 + \e^{-d/2} |u|_{s_0 + 1})   | u |_{s+1},$$
  and \eqref{1} holds  with 
 \begin{equation} \label{cphi} {\bf C}_0 = 1 + \sum_{1 \leq j \leq d} C_{A_j}\big( \e^{-d/2} | u |_{s_0}\big)(1 + \e^{-d/2} |u|_{s_0 + 1}) .
 \end{equation} 
  Besides,
  $$ \begin{aligned}
  \| (\Phi^\e)'(u) v \|_s & \leq \sum_{1\leq j \leq d} | (A_j'(u) \cdot v) \e\d_{x_j} u |_s + | A_j(u) \e \d_{x_j} v |_s + |v|_{s+1} \\
   & \leq \sum_{1 \leq j \leq d} C_{A_j,A'_j}(|u|_{L^\infty}) \big( \e^{-d/2} |v|_{s_0 + 1} | u |_{s+1} + (1 + \e^{-d/2} | u |_{s_0+1}) |v|_{s+1}\big) \end{aligned}$$
 and \eqref{1'} holds $\g_0 = \g_1 = d/2$ and 
  $$ {\bf C}_1 = 1 + \sum_{1 \leq j \leq d} C_{A_j,A'_j}\big( \e^{-d/2} | u |_{s_0}\big)(1 + \e^{-d/2} |u|_{s_0 + 1}).$$  %
  The bound for $(\Phi^\e)''$ is similar. We conclude that Assumption {\rm \ref{ass1}} holds with $\g_0 = \g_1 = d/2.$
 \end{ex}

  \begin{rem} In connection with \eqref{embed-norm}, where the embedding constant blows up to $+\infty$ as $\e$ decreases to $0,$  
the ``constants" ${\bf C}_j$ in Assumption {\rm \ref{ass1}} are \emph{not} assumed to be non-decreasing in their arguments; in Example {\rm \ref{ex3}} this is reflected in \eqref{cphi}.
\end{rem}

 \begin{ex} \label{ex4} Given $T > 0,$ consider the functional spaces
 $$ E_s =  C^1([0, T], H^{s-1}(\R^d)) \cap C^0([0, T], H^{s}(\R^d)), \quad  F_{s-1}  =  C^0([0, T], H^{s-1}(\R^d)) \times H^{s}(\R^d),$$
 with norms
 $$ | u |_{s} := \sup_{0 \leq t \leq T} ( \| \e \d_t u(t) \|_{H^{s-1}_\e} + \| u(t) \|_{H^s_\e} ), \quad  \| (f_1,f_2) \|_{s-1}  := \sup_{0 \leq t \leq T} \| f_1(t)\|_{H^{s-1}_\e} + \| f_2\|_{H^{s}_\e},$$
 where the weighted Sobolev $\| \cdot \|_{H^s_\e}$ norms are defined in \eqref{hse}.

 Let $s_0$ and $s_1$ be given such that $d/2 < s_0 < 1 + d/2 < s_1,$ with $s_0 + 2 \leq s_1.$ Let a bounded family $(f^\e)_{\e \in (0,1)} \subset F_{s_1 - 1};$ meaning $f^\e \in F_{s_1 - 1}$ for all $\e,$ with $\sup_{\e \in (0,1)} \| f^\e \|_{s_1 - 1} < \infty.$
 
   Consider the family of maps defined by 
  \begin{equation} \label{phie} \Phi^\e(u) := \Big(\e \d_t u + \sum_{1 \leq j \leq d} A_j( u) \e \d_{x_j} u - f_1^\e\, , \, u_{|t = 0} - f_2^\e\Big),
  \end{equation}
 where the maps $A_j: \, u \in \R^n \to A_j(u) \in \R^{n \times n}$ are smooth. Then, $\Phi^\e$ maps $E_s$ to $F_{s-1}$ for $s_0 + 1 \leq s \leq s_1,$ and we can check exactly as in Example {\rm \ref{ex3}} that Assumption {\rm \ref{ass1}} holds with $\g_0 = \g_1 =  d/2.$
  \end{ex}

\begin{rem}\label{rem:ass1-shift} Assumption {\rm \ref{ass1}} is stable by shifts that are both smooth and small, in the following sense: if $\Phi^\e$ satisfy Assumption {\rm \ref{ass1}}, with associated parameters $m,$ $s_0,$ $s_1,$ and $\g_0,$ and if $(a^\e)_{\e \in (0,1)} \subset E_{s_1},$ with $\sup_{\e \in (0,1)} \e^{-\g_0} |a^\e|_{s_0+m} < \infty,$ then $\tilde \Phi^\e := \Phi^\e(a^\e + \cdot)$ satisfies Assumption {\rm \ref{ass1}} with the same parameters. \end{rem}

\subsection{Second assumption: tame inverse bounds}

Our second and key assumption states that if $u$ is small enough in some ``low" norm, then $(\Phi^\e)'(u)$ has a right inverse, with a right inverse bound \eqref{2} that is possibly stiff with respect to $\e$ and may show a loss of derivatives: 
 
\begin{ass} \label{ass2} For some $r \geq 0,$ $r' \geq 0,$ there holds $s_0 + m \leq s_1 - \max(m,r),$ and for some $\g \geq 0,$ $\k \geq 0,$ for all $s$ such that $s_0 +m \leq s \leq s_1 - \max(m,r),$ 
   for all $u \in E_{s + r}$ such that
\begin{equation} \label{cond1}
 | u |_{s_0 + m + r} \lesssim \e^{\gamma},
 \end{equation}
 the map $(\Phi^\e)'(u): E_{s+m} \to F_s$ has a right inverse $\Psi^\e(u):$
  $$ (\Phi^\e)'(u) \Psi^\e(u) = {\rm Id}: \quad F_s \to F_s,$$
   satisfying, for all $g \in F_{s+r'},$
 \begin{equation} \label{2}
  | \Psi^\e(u) g|_s \leq {\bf C}  \e^{-\k} ( \|g\|_{s_0 + m + r'} | u|_{s+r} + \|g\|_{s + r'}),
 \end{equation}
where
 ${\bf C} = {\bf C}(\e,|u|_{s_0 + m + r})$ satisfies
 $$ \sup \big\{{\bf C}, \quad \e \in (0,1), \, |u|_{s_0 + m + r} \lesssim \e^\g \big\} < \infty.$$
 \end{ass}

 \begin{rem}[On stiffness of the right-inverse bound \eqref{2}] The right-inverse bound \eqref{2} is stiff with respect to $\e$ if $\k > 0.$ The case $\k = 0$ corresponds to no loss in $\e$ and is covered for instance by the result of Alvarez-Samaniego and Lannes {\rm \cite{AL}.} 
 \end{rem}

 \begin{rem}[On losses of derivatives in the right-inverse bound \eqref{2}] The loss of derivatives is parameterized by $r$ and $r'.$ Estimate \eqref{2} states indeed that we can solve the linearized equation $(\Phi^\e)'(u) v = g,$ with a bound for the solution $v = \Psi^\e(u) g$ that gives a control of the low norm $|v|_s$ in terms of high norms, $|u|_{s+r}$ and $\| g \|_{s+r'},$ of the background and source.
 
  The case $r = r' = 0$ corresponds to no loss and can typically be handled by Picard iteration, as in the classical existence proof for quasilinear symmetric systems.

  Estimate \eqref{2} in Assumption 
{\rm \ref{ass2}} is \emph{tame} with respect to $r$ and $r'.$ 
This is essential:
one may check that the proof of Theorem {\rm \ref{th1}} (given in Section {\rm \ref{sec:pf}}) collapses if it is not.

The distinction between $r$ and $r'$ is somewhat illusory, since
we can always redefine
 $F_s$ and $m,$ so that $r' = 0.$ 
We also note that the proof with $r' \neq 0$ is the same as with $r' = 0$.
 \end{rem}

 A simple example of a family of maps satisfying Assumption \ref{ass2} is given by systems of symmetric quasi-linear equations in a semi-classical setting, as detailed in the following Example. (This example is meant to test our theory in the favorable context of symmetric quasi-linear systems, where a Lax iteration scheme is certainly preferable to a Nash-Moser scheme.)

\begin{ex} \label{sec:app0} Consider the functional spaces and the family of maps introduced in Example {\rm \ref{ex4}}. Again, $T > 0$ and $d/2 < s_0 < 1 +  d/2 < s_1,$ are given, with $s_1$ measuring the regularity of $(f_1^\e, f_2^\e),$ and $s_0 + 2 \leq s_1.$  Then, as discussed in Example {\rm \ref{ex4}}, Assumption {\rm \ref{ass1}} holds with $\g_0 = \g_1 = d/2$ and $m =1.$ 

 Given $s$ such that $s_0 + 1 \leq s \leq s_1 -1,$ and $u \in E_{s+1},$ $g = (g_1,g_2) \in F_{s},$ consider the equation $(\Phi^\e)'(u) v = g,$ corresponding to the linearized initial-value problem   
 \begin{equation} \label{lin-ivp} \left\{ \begin{aligned} \e \d_t v + \sum_{1 \leq j \leq d} A_j( u) \e \d_{x_j} v & = g_1 - \sum_{1 \leq j \leq d} A'_j(u) v \e \d_{x_j} u, \\ v_{|t = 0} & = g_2. \end{aligned}\right.
 \end{equation}
 Assume that the maps $A_j$ take values in the symmetric matrices. Then, by the classical linear hyperbolic theory, there exists a unique $v \in E_s$ solution of \eqref{lin-ivp}. We now show that $v$ satisfies an estimate of the form \eqref{2}.
 
 The classical commutator estimate gives, for $0 \leq |\a| \leq s$ and $w \in H^s,$
  \begin{equation} \label{fried} 
 \Re e\, \Big( (\e\d_x)^\a A_j(u) \e \d_{x_j} w, (\e \d_x)^\a w\Big)_{L^2} \leq C(|u|_{L^\infty})\Big( |\e\d_x u|_{L^\infty} \| w \|_{H^s_\e}^2 + \| u \|_{H^s_\e} |w|_{L^\infty} \| w \|_{H^s_\e}\Big),
  \end{equation}
  where the weighted Sobolev norm $\| \cdot \|_{H^s_\e}$ is defined in \eqref{hse}. Besides, by Remark {\rm \ref{rem:Moser}},
 \begin{equation} \label{fried2}
  \| A'_j(u) w \e \d_{x_j} u \|_{H^s_\e} \leq C(|u|_{L^\infty}) \Big( |\e \d_{x_j} u |_{L^\infty} \big( \| w \|_{H^s_\e} + \| u \|_{H^s_\e} |w|_{L^\infty}\big) + \| u \|_{H^{s+1}_\e} | w |_{L^\infty} \Big).
  \end{equation}
  If we now use \eqref{embed-norm} with $N = 0,1,$ together with estimates \eqref{fried} and \eqref{fried2}, we find that the solution $v$ to \eqref{lin-ivp} satisfies
  \begin{equation} \label{fried3}
  \begin{aligned} \| v(t) \|_{H^s_\e}^2  \leq \| g_2 \|^2_{H^s_\e} & + \int_0^t \e^{-1} \| g_1 \|_{H^s_\e} \| v\|_{H^s_\e} \, dt' \\ &  + \int_0^t  C\big(\e^{-1-d/2} \|u\|_{H^{s_0+1}_\e}\big) \|v\|_{H^s_\e}^2 \, dt'  \\ & + \int_0^t  C\big(\e^{-d/2} \|u\|_{H^{s_0+1}_\e}\big) \e^{-1-d/2} \|u\|_{H^{s+1}_\e} \|v\|_{H^{s_0}_\e} \| v\|_{H^s_\e} \, dt'.\end{aligned}
  \end{equation}
 We now restrict to a background $u \in H^{s+1}$ satisfying 
 \begin{equation} \label{s0+2} |u|_{s_0+2} \lesssim \e^{1+d/2}.
 \end{equation}
  Then, by Gronwall's Lemma, estimate \eqref{fried3} used with $s = s_0,$ implies the bound
 \begin{equation} \label{vs0} \max_{[0,T]} \|v\|_{H^{s_0}_\e} \lesssim \|g_2\|_{H^{s_0}_\e}  + \e^{-1} \max_{[0,T]} \|g_1\|_{H^{s_0}_\e},
 \end{equation}
 which we use back in \eqref{fried3} to obtain
 \begin{equation} \label{fried4} 
 \begin{aligned}
 \max_{[0,T]} \| v\|_{H^s_\e} \lesssim \|g_2\|_{H^s_\e} & + \e^{-1} \max_{[0,T]} \| g_1 \|_{H^s_\e} \\ & + \e^{-1-d/2} \big( \|g_2\|_{H^{s_0}_\e}  + \e^{-1} \max_{[0,T]} \|g_1\|_{H^{s_0}_\e}\big) \max_{[0,T]} \|u\|_{H^{s+1}_\e}.
  \end{aligned}
 \end{equation}
 By using equation \eqref{lin-ivp}{\rm (i)} directly, we find the bound
  \begin{equation} \label{fried5}
   \| \e \d_t v \|_{H^{s-1}_\e} \lesssim C(|u|_{L^\infty},|\e\d_x u|_{L^\infty})\Big( \| v \|_{H^s_\e} + \| u \|_{H^s_\e}( |v|_{L^\infty} + |\e\d_x v |_{L^\infty})\Big).
   \end{equation}
 Note in \eqref{fried5} the occurence of $|\e\d_x v|_{L^\infty},$ which cannot be controlled by \eqref{vs0}. This forces us to go back to \eqref{fried3} with $s = s_0 + 1.$ At this point we make full use of \eqref{s0+2}, while a bound for $|u|_{s_0+1}$ sufficed in order to estimate $\| v \|_{H^s_\e},$ and obtain
 \begin{equation} \label{vs0+1}
  \max_{[0,T]} \| v \|_{H^{s_0 + 1}_\e} \lesssim \|g_2\|_{H^{s_0+1}_\e}  + \e^{-1} \max_{[0,T]} \|g_1\|_{H^{s_0+1}_\e}.
  \end{equation}
 Combining \eqref{fried4}, \eqref{fried5} and \eqref{vs0+1}, we obtain that, if $u$ satisfies \eqref{s0+2}, then there holds
 $$ |v|_s \lesssim \e^{-2 - d/2} \| g \|_{s_0+1} |u|_{s+1} + \e^{-1} \| g \|_{s}.$$
We can conclude that if for all $j,$ $A_j$ is symmetric, then the family of maps and functional spaces defined in Example {\rm \ref{ex4}} satisfy Assumption {\rm \ref{ass2}} with $\g = 1 + d/2,$ $\k = 2 + d/2,$ $r = 1$ and $r' = 0.$ 
  \end{ex}

\begin{rem} \label{rem:ass2-shift} If $\Phi^\e$ satisfies Assumption {\rm \ref{ass2}} with parameters $\g,$ $\k,$ $r,$ $r',$ given a family $(a^\e)_{\e \in (0,1)} \subset E_{s_1},$ if $\sup_{\e \in (0,1)} \e^{-\g} |a^\e|_{s_0 + m + r} < \infty,$ then $\tilde \Phi^\e := \Phi^\e(a^\e + \cdot)$ satisfies Assumption {\rm \ref{ass2}}, with the same parameters as $\Phi^\e.$ 
\end{rem}

\subsection{Third assumption: existence of an approximate solution}

 We consider a family of maps $\Phi^\e: E_s \to F_{s-m}$ satisfying Assumptions \ref{ass1} and \ref{ass2}, with associated parameters $m,$ $s_0, s_1,$ and $\g_0,$ $\g_1,$ $\g,$ $r, r',$ and $\k.$  

\begin{ass} \label{ass-wkb} Let $k$ such that 
 \begin{equation} \label{ass-k}
  \max(\k + \g_0, 2\k + \g_1, \k + \g) < k.
  \end{equation}
  For some positive function $\bar p = \bar p(m,r,r',\g_0, \g_1, \g, \k, k) \geq 2 m + \max(r,r')$ specified in Remark {\rm \ref{rem3}}, there holds
  \begin{equation} \label{low-s1}
  s_0 + m + \max(r,r') + \bar p < s_1,
  \end{equation}
  and, for some $s$ satisfying 
 \begin{equation} \label{ass-s}
  s_0 + m + \max(r,r')  \leq s < s_1 - \bar p,
  \end{equation}
  there holds
  \begin{equation} \label{wkb3}
 \|\Phi^\e(0)\|_{s} \lesssim \e^{k}.\end{equation}
  \end{ass}

 We first comment on \eqref{low-s1}:

\begin{rem} \label{rem:reg} In Example {\rm \ref{sec:app0}}, the index $s_1$ measures the regularity of the source $f_1^\e$ and the initial datum $f_2^\e;$ in this view inequality \eqref{low-s1} should be understood as a regularity requirement on the data. In particular, as discussed in Remark {\rm \ref{rem3}}, as $k$ approaches from above the limiting value $\max(\k + \g_0, 2 \k + \g_1,\k + \g),$ the parameter $\bar p$ blows up to $+\infty,$ meaning that we require the data to be infinitely regular in this limit.  
\end{rem}
  
  In our examples, inequality \eqref{wkb3} reflects the existence of a family of approximate solutions to $\Phi^\e = 0:$ 
 
\begin{rem} \label{r-app} Consider a family of maps $\Phi^\e \in C^2(E_s,F_{s-m})$ satisfying Assumptions {\rm \ref{ass1}} and {\rm \ref{ass2}}, and an associated family of approximate solutions $u_a^\e \in E_{s_1}$ to the equations $ \Phi^\e = 0,$ in the sense that $ \| \Phi^\e(u_a^\e) \|_{s} \lesssim \e^k,$
 with $k$ and $s$ satisfying \eqref{ass-k}-\eqref{ass-s}.
 Then, the maps $\tilde \Phi^\e := \Phi^\e(u_a^\e + \cdot)$ satisfy Assumption {\rm \ref{ass-wkb}}. 
\end{rem}

 For quasilinear systems, small initial data give crude examples of approximate solutions:

\begin{ex} \label{ex:rescaled} Consider the initial-value problem 
 \begin{equation} \label{ql-ivp} \d_t u + \sum_{1 \leq j \leq d} A_j(u) \d_{x_j} u = 0, \qquad u_{|t = 0} = \e^\s a(x), 
 \end{equation}
 and the associated family of maps $\Phi^\e$ defined in \eqref{phie}, with $(f_1^\e, f_2^\e) \equiv (0,\e^\s a).$ Assume that $a \in H^{s_1},$ with $s_1$ satisfying \eqref{low-s1}. As described in Examples {\rm \ref{ex4}} and {\rm \ref{sec:app0}}, if $A_j$ is symmetric for all $j,$ then Assumptions {\rm \ref{ass1}} and {\rm \ref{ass2}} hold for $\Phi^\e$ in the functional spaces $E_s,$ $F_s$ introduced in Example {\rm \ref{ex4}}, for any $T > 0,$ for $s$ satisfying \eqref{ass-s}. 
 
  If the maps $u \to A_j(u)$ satisfy the bounds $|A_j(u)| \lesssim |u|^\ell,$ for $\ell \geq 0,$ in particular if they are $\ell$-homogeneous, then there holds $\| \Phi^\e(\e^\s a) \|_{s_1-1} \lesssim \e^{\s(1 + \ell) + 1}.$ Using Remark {\rm \ref{r-app}} above, and the specific values of $\g_0,$ $\g_1,$ $\g$ and $\k$ given in Examples {\rm \ref{ex4}} and {\rm \ref{sec:app0}}, this implies that if $\s(1 + \ell) > 3(1 + d/2),$ then Assumption {\rm \ref{ass-wkb}} is satisfied by $\tilde \Phi^\e := \Phi^\e(\e^\s a + \cdot).$ 
  
\end{ex}

\subsection{Results}

  Our main result gives existence in $E_{s+m},$ where $s$ satisfies \eqref{ass-s}, of a solution to equation \eqref{0}. 
  
  \begin{theo}[Existence] \label{th1} Under Assumptions {\rm \ref{ass1},} {\rm \ref{ass2}} and {\rm \ref{ass-wkb}}, for $\e$ small enough, there exists a real sequence $\theta_j^\e,$ satisfying $\theta_j^\e \to +\infty$ as $j \to +\infty$ and $\e$ is held fixed, such that the sequence
 \begin{equation} \label{iteration} u^\e_0 := 0, \qquad u^\e_{j+1} := u_j^\e + S_{\theta_j^\e} v_j^\e, \qquad v_j^\e := - \Psi^\e(u_j^\e) \Phi^\e(u_j^\e),
 \end{equation}
  is well defined and converges, as $j \to \infty$ and $\e$ is held fixed, to a solution $u^\e$ of \eqref{0} 
in $s +m$ norm, with $s$ as in \eqref{ass-s}, which satisfies the bound
 \begin{equation} \label{est-th}
 | u^\e |_{s} \lesssim \e^{k -\k}. 
 \end{equation}
\end{theo}

In applications (Sections \ref{sec:app1} and \ref{sec:app2}), we apply Theorem \ref{th1} to a map $\Phi^\e(u_a^\e + \cdot),$ with the notation of Remark \ref{r-app}, so that in practice Theorem \ref{th1} is not a result about small solutions: the smallness condition \eqref{est-th} bears on the perturbation variable $u,$ the full solution to $\Phi^\e = 0$ being $u_a^\e + u.$ The smallness condition \eqref{wkb3} is an accuracy condition bearing on the approximate solution $u_a^\e;$ we show in Remark \ref{rem:newt} that this condition is sharp in the usual implicit function theorem setting without losses in derivatives.

We supplement the above existence result by the following local uniqueness Theorem. Contrary to Theorem \ref{th1}, Theorem \ref{th2} does not rely on a Nash-Moser iterative scheme.

 \begin{theo}[Local uniqueness] \label{th2}
Under Assumptions {\rm \ref{ass1},} {\rm \ref{ass2}} and {\rm \ref{ass-wkb}}, 
for $\e$ small enough, 
if $(\Phi^\e)'$ is invertible, i.e., $\Psi^\e$ is also a left inverse, then
the solution described in Theorem {\rm \ref{th1}} is unique in a ball of radius 
$o(\eps^{\max(\k + \g_1,\gamma_0,\gamma)})$ in 
$s_0 +2m + r'$ norm.
More generally, if $\hat u^\e$ is a second solution within this ball, then
$(\hat u^\e- u^\e)$ is approximately tangent to $\kernel (\Phi^\e)'(u^\e)$, in the
sense that its distance in $s_0$ norm from $\kernel (\Phi^\e)'(u^\e)$ is
$o(|\hat u^\e-u^\e|_{s_0 })$.
In particular, if $\kernel (\Phi^\e)'(u^\e)$ is finite-dimensional,
then  $u$ is the unique solution in the ball satisfying the additional
``phase condition''
\be\label{phasecond}
\Pi_{u^\e} (\hat u^\e-u^\e)=0,
\ee
where 
$\Pi_{u^\e}$ is any uniformly bounded projection 
onto $\kernel (\Phi^\e)'(u^\e).$ (In a Hilbert space, any orthogonal projection onto $\kernel (\Phi^\e)'(u^\e)$.)

\end{theo}

 In a non-Hilbertian context, the existence of such a projection $\Pi_{u^\e}$  is discussed in Remark \ref{rem:Pi} below.

\subsection{Remarks}

 We first remark that the proofs use only part of the information contained in \eqref{2} and \eqref{1''}:
\begin{rem} \label{rem0'} 
An examination of the proofs of Theorems {\rm \ref{th1}} and {\rm \ref{th2}}
reveals that for existence we require  
estimate \eqref{2} only for $f$ in the image 
of $\Phi^\eps$ or $(\Phi^\eps)''$,
since $\Psi^\eps$ is estimated only in composition
with one or the other of these operators, while for uniqueness we need only
the estimate for $\Psi^\eps(u)(\Phi^\eps)''(u)$ that would result by
composing estimates \eqref{2} and \eqref{1''}.
\end{rem}

\medskip

Next we remark that the approximation rate is sharp by comparison with the standard Newton scheme:
\begin{rem}  \label{rem:newt}
For $\g_0, \g \leq \k + \g_1$, corresponding to a critical value $k_c = 2\k + \g_1$ in \eqref{ass-k}, Theorem {\rm \ref{th1}} states that
a loss of 
$\e^{-\k}$ in the linear estimates means that, 
with the notation of Remark {\rm \ref{r-app},}
$\| \Phi^\e(u^\e_a)\|_s\lesssim \eps^{2 \k + \g_1 + \eta},$ any
$\eta>0,$ is the accuracy needed on the approximate solution. 

 This condition is sharp even for convergence of a standard
Newton iteration scheme 
 $$u^\eps_{n+1} :=u^\eps_n- \Psi^\eps(u^\eps_n) \Phi^\eps(u^\eps_n), \qquad u_0^\e := u_a^\e, \quad |u_a^\e|_{s+m} \lesssim \e^{\g_1},$$
  for problems with no loss of derivatives ($r = r' = 0$),
corresponding by the computation
$$ \begin{aligned}
 \| \Phi^\eps(u^\eps_{1}) \|_s & = \big\| \int_0^1 (1 -t) (\Phi^\e)''(u_0^\e + t (u_1^\e - u_0^\e)) \cdot (u_1^\e - u_0^\e, u_1^\e - u_0^\e) \big\|_s \\  & \lesssim \big( \e^{-2 \g_1} (\| u_0^\e\|_{s}  +  \| u_1^\e - u_0^\e\|_{s}) + \e^{-\g_1} \big)\| u_1^\e - u_0^\e \|_{s}^2 \\ & \lesssim \e^{-\g_1} \| u_1^\e - u_0^\e \|_{s}^2 \\ &
 \lesssim \e^{-\g_1} |\Psi^\eps(u^\eps_0) \Phi^\eps(u^\eps_0)|_s^2
\\ & \lesssim 
\eps^{-(2 \k + \g_1)}\|\Phi^\eps(u^\eps_0)\|_s^2 \\ & \lesssim \eps^{\eta} \|\Phi^\eps(u^\eps_0)\|_s 
\end{aligned}$$
in the case $m = 0$ to the condition that error $(\|\Phi^\eps(u^\e_n)\|_s)_{n \in \N}$ decreases at the first step.
\end{rem}

\medskip

 We make precise the parameter $\bar p$ that appears in Assumption \ref{ass-wkb}: 

\begin{rem} \label{rem3} 
 From \eqref{ens-cond}, we find that $\bar p$ is
\begin{equation} \label{barp} \bar p = m + \inf_{N > N_0} \frac{(N+1)(m + \max(r,r') + M)}{1 - \frac{\k}{k} - M},
\end{equation}
with 
$$N_0 :=  \frac{\k + k m'}{k - (2\k + \g_1)},   \, \, M := \max\left(\frac{\g_0}{k}, \frac{\g}{k}, \frac{1}{2} \left(1 +  \frac{\g_1}{k} + \frac{m'}{N} + \frac{\k}{kN}\right)\right), \, \, m' := \max(m+r',r).$$
 We observe the following asymptotic behavior as $k$ approaches from above the critical value $k_c := \max(\k + \g_0, 2 \k + \g_1, \k + \g)$ given in \eqref{ass-k}:
 \begin{itemize}
 \item If $k_c = 2 \k + \g_1,$ then $\bar p$ blows up like $(k-k_c)^{-2}$ as $k \downarrow k_c.$
 \item If $k_c = \k + \g_0,$ or $k_c = \k + \g,$ then $\bar p$ blows up like $(k-k_c)^{-1}$ as $k \downarrow k_c.$
 \end{itemize} 
\end{rem}

 The phase condition \eqref{phasecond} can in some situations be made explicit: 
 
\begin{rem} \label{r-uniqueness} Let $\Pi_u$ be a bounded projection onto $\ker (\Phi^\e)'(u),$ as in \eqref{phasecond}. If the map $(u,v) \to \Pi_{u} v$ is continuous in $E_{s_0} \times E_{s_0},$ uniformly in $\e,$ then the implicit phase condition \eqref{phasecond} can be replaced by the explicit 
 $$
\Pi_{0} (\hat u^\e-u^\e)=0.
$$
See {\rm \cite{TZ2},} Section 2, for related discussions of uniqueness up to phase conditions. 
\end{rem}

\medskip

 We discuss the existence of the projection mentioned in Theorem \ref{th2}:

\begin{rem} \label{rem:Pi} We first remark that if $\kernel (\Phi^\e)'(u^\e)$ is finite-dimensional, then a bounded projection exists by the Hahn--Banach Theorem; see, e.g., {\rm \cite{RY}.}

 In the infinite-dimensional case, we note that
a projection $\Pi$ onto a subspace $S$ of a Banach space
is bounded if and only if the distance from $s\in S$ to
$\kernel \,\Pi$ is greater than or equal to $|s|/C$ for some uniform $C>0$.
(Indeed, 
$|s|=|\Pi (s-t)|\le C|s-t|$ for all $s\in S$, $t\in \kernel \,\Pi$
is equivalent to the statement that $\Pi$ is bounded, 
since $s-t$ runs over the entire Banach space as $s$ and $t$
are varied.)

 This implies that if there is an isometry between
spaces $F^\eps_s$ and a common set of spaces $F^0_s$,
and if that $\kernel (\Phi^\eps)'(u^\eps)$, considered
(under mapping by this isometry) as a subset
of $F^0_s$ is finite-dimensional, with a limit as $\eps\to 0,$
then, there exist a family of projections $\Pi^\eps$ onto $\kernel
(\Phi^\eps)'(u^\eps)$ that are uniformly bounded with respect to $\eps$
in each $F^\eps_s$, for $\eps$ sufficiently small.

 Indeed, by the above note (second paragraph of the present Remark), there exists a bounded projection $\Pi^0$
onto the limit as $\eps\to 0$ of $\kernel (\Phi^\eps)'(u^\eps)$.
Denote by $\tilde F:=\kernel \Pi^0$ the associated complementary subspace.
Defining $\Pi^\eps$ to be the projection along $\tilde F$ onto
$\kernel (\Phi^\eps)'(u^\eps)$, we find by the Hahn-Banach theorem,
compactness of the intersection of the unit ball with
$\kernel (\Phi^\eps)'(u^\eps)$, and continuity,
that $\Pi^\eps$ is bounded for $\eps$ sufficiently small.
\end{rem}

\medskip

 We finally describe a somewhat artificial application for orientation:
 
\begin{ex} \label{ex:ql} Consider the quasilinear initial-value problem \eqref{ql-ivp}, with associated maps $\Phi^\e.$ The datum $\e^\s a$ is assumed to belong to $H^{s_1},$ with $s_1$ as in \eqref{low-s1}. 

 By Examples {\rm \ref{ex4}} and {\rm \ref{sec:app0}}, if the matrices $A_j$ are symmetric, then Assumptions {\rm \ref{ass1}} and {\rm \ref{ass2}} hold in $E_s,$ $F_s,$ for any $T >0,$ for $s$ satisfying \eqref{ass-s}. 
 
 By Remarks {\rm \ref{rem:ass1-shift}} and {\rm \ref{rem:ass2-shift}}, if the initial datum is small enough: $\s > d/2,$ then Assumptions {\rm \ref{ass1}} and {\rm \ref{ass2}} still hold, with the same parameters, for the translated maps $\tilde \Phi^\e := \Phi^\e(\e^\s a + \cdot).$
 
 By Example {\rm \ref{ex:rescaled}}, if $\s > \s_c := \max(d/2,3(1 + \ell)^{-1} (1 + d/2)),$ then the translated maps $\tilde \Phi^\e$ satisfy in addition Assumption {\rm \ref{ass-wkb}}.
 
 We can conclude that Theorem {\rm \ref{th1}} yields existence in $C^0([0,T], H^s(\R^d))$ of a solution to \eqref{ql-ivp}. If $\ell = 1,$ the smallness condition for the initial datum if $\s > 3/2 + 3d/4.$
 
  We thus partly recovered a classical small-amplitude existence result of the quasilinear hyperbolic theory. Note that, as mentioned in Remark {\rm \ref{rem:reg}}, if $\s - \s_c$ is small, then $\bar p$ is large, hence $s,$ satisfying \eqref{ass-s}, is much smaller than $s_1,$ meaning that the solution is much less regular than the datum. 
 
 \end{ex}
 
 For quasilinear symmetric systems, the Lax iteration scheme gives an existence result with no smallness assumption on the datum, and with the sharp regularity criterion $s > 1 + d/2.$ In this view, it is much better suited for the resolution of quasilinear symmetric systems than the Nash-Moser scheme described above.

 \begin{section}{Proofs of Theorems \ref{th1} and \ref{th2}}\label{sec:pf}
 
 We write $\Phi$ for $\Phi^\e,$ $\theta_j$ for $\theta_j^\e,$ etc. Let $\theta_0$ such that
  \begin{equation} \label{cond1-0}
   \theta_0^{-\a} \leq \e^{\max(\g_0, \g_1, \g)},
   \end{equation}
   for some $\a > 0$ to be chosen later. Introduce the family of inequalities ${\cal C}_1(j),$ for $j \in \N,$ 
     $$ {\cal C}_1(j;q,\a): \qquad | v_j|_{s + q} \lesssim \theta_j^{-\a}$$
 depending on $\a$ and some $q \geq m$ and $s$ to be chosen later. We assume that Assumptions \ref{ass1}, \ref{ass2}, and \ref{ass-wkb} hold, and start by proving three Lemmas.
\begin{lem} \label{l1} Assume that $s_0 < s \leq s_1 - q,$ and 
 \begin{itemize}
 \item the sequence $u_j$ is well defined,
 \item $\lim_{j \to +\infty} \|\Phi(u_j)\|_s = 0,$
 \item condition ${\cal C}_1(j;q,\a)$ holds for all $j,$
 \item the series $\theta_j^{-\a}$ is convergent, with
 \begin{equation} \label{cv-series}
  \sum_{j=0}^{+\infty} \theta_j^{-\a} \lesssim \theta_0^{-\a}.
  \end{equation}
  \end{itemize}
 Then $u_j$ converges, in $s+q$ norm, to a solution of \eqref{0} which satisfies
  \begin{equation} \label{e1} | u |_{s+q} \lesssim \theta_0^{-\a}.\end{equation}
\end{lem}

\begin{proof}
 If ${\cal C}_1(j)$ holds for all $j,$ then the sequence $u_j$ converges, in $s+q$ norm, to $u \in E_{s+q},$ and we have the estimate 
 \begin{equation} \label{e1.1} | u_j |_{s+q} \lesssim  \sum_{j=0}^{j-1} \theta_j^{-\a},
 \end{equation}
which implies \eqref{e1}. There holds
 \begin{equation} \label{lastone} \|\Phi(u)\|_s \leq \|\Phi(u_j)\|_s + \big\| \int_0^1 \Phi'(u_j + t(u - u_j)) \cdot (u - u_j) \, dt \, \big\|_s,
 \end{equation}
 and the first term in the upper bound tends to $0$ as $j \to \infty.$ We note that, by \eqref{cv-series}, \eqref{e1}, \eqref{e1.1}, and \eqref{cond1-0}, there holds $|u_j|_{s+m} + |u |_{s + m} \lesssim \e^{\g_0}.$ Hence, by the tame direct bound \eqref{1'} in Assumption \ref{ass1},  
 $$ \begin{aligned} \big\| \int_0^1 \Phi'(u_j + t(u - u_j)) \cdot (u - u_j) \, dt \, \big\|_s  \lesssim |u - u_j|_{s+m},\end{aligned}$$
 The upper bound tends to 0 as $j \to +\infty.$ With \eqref{lastone}, this implies that $u$ solves \eqref{0}.
\end{proof}
   
Let $p$ such that
 \begin{equation} \label{con-p}
   q + \max(r, r' + m) \leq p, \qquad s_0 + m + \max(r,r') + p  \leq s_1.
   \end{equation}
 Introduce the family of inequalities ${\cal C}_2(j),$ for $j \in \N$ and $N \geq 0:$ 
$$ {\cal C}_2(j;q,\a,p,N): \qquad \left\{ \begin{aligned} |u_j|_{s + q} & \lesssim & \theta_0^{-\a}, \\ \| \Phi(u_j)\|_{s} & \lesssim  &\theta_j^{-1}, \\ |u_j|_{s + p} & \lesssim &  \theta_j^N.\end{aligned}\right.$$

\begin{lem} \label{l2} Assume that $s_0 + m + \max(r,r') \leq s \leq s_1 - p,$ and 
\begin{itemize}
 \item for all $j' \leq j,$ $u_{j'}$ is well defined and condition ${\cal C}_1(j')$ holds,
 \item there holds 
  \begin{equation} \label{con1} \sum_{j'=0}^j \theta_j^{-\a} \lesssim \theta_0^{-\a},
  \end{equation}
  \item condition ${\cal C}_2(j)$ holds, with parameters satisfying  \begin{equation} \label{con2} \theta_j^{m - q - \a} + \e^{-\g_0} \theta_j^{-2\a} \leq \theta_{j+1}^{-1},\end{equation}
  \begin{equation} \label{con3} 
   \e^{-\k} \theta_j^{\max(m+r',r)} \theta_j^N \leq \theta_{j+1}^N.
   \end{equation}
   \end{itemize}
  Then $v_{j+1}$ is well defined in $E_{s+q}$ and ${\cal C}_2(j+1)$ holds.
 \end{lem}

\begin{proof} If conditions ${\cal C}_1(j')$ hold for all $j' \leq j$ and if \eqref{con1} holds, then
 \begin{equation} \label{jplusone}
  |u_{j+1} |_{s+q} \lesssim \theta_0^{-\a}.
  \end{equation}
Bound \eqref{jplusone} is ${\cal C}_2(j+1)$(i). Besides, \eqref{jplusone} and \eqref{cond1-0} imply that $u_{j+1}$ also satisfies \eqref{cond1}, so that, by ${\cal C}_2(j+1)$(iii), the first bound in \eqref{con-p} and the tame inverse bound \eqref{2} in Assumption \ref{ass2}, $v_{j+1}$ is defined in $E_{s+q}.$ 

 To prove ${\cal C}_2(j+1)$(ii), we use the fact that \eqref{iteration} is almost a Newton's scheme:
 $$ \| \Phi(u_{j+1})\|_{s} \leq E_1 + E_2,$$
 where $E_1$ is the error due to the regularization:
 $$ E_1 = \| \Phi'(u_j) \cdot (S_{\theta_j} v_j - v_j) \|_s,$$
 and $E_2$ is the error due to the scheme: 
 $$ E_2 = \left\| \int_0^1 (1 - t) \Phi''(u_j + t S_{\theta_j} v_j) \cdot (S_{\theta_j} v_j, S_{\theta_j} v_j) \, dt \right\|_s.$$
Conditions $({\cal C}_1(j'))_{1\leq j'\leq j-1},$ together with \eqref{cond1-0} and \eqref{con1}, imply $ |u_j|_{s+m} \lesssim \e^{\g_0}.$ Together with the tame direct bound \eqref{1'} in Assumption \ref{ass1}, this gives 
 $$ E_1  \lesssim  |S_{\theta_j} v_j - v_j|_{s +m },$$
and with \eqref{3} and ${\cal C}_1(j),$ 
  \begin{equation} \label{bd-E1} E_1  \lesssim  \theta_j^{m - q - \a}.
 \end{equation}
By ${\cal C}_1(j),$ \eqref{4}, and \eqref{cond1-0}, there holds $|S_{\theta_j} v_j|_{s+m} \lesssim \e^{\g_0}.$ With the tame direct bound \eqref{1''} in Assumption \ref{ass1}, this gives
$$ E_2  \lesssim \e^{-2 \g_1} |S_{\theta_j} v_j|_{s_0 + m}^2 \left( |u_j|_{s+m} + |S_{\theta_j} v_j|_{s+m}\right) + \e^{-\g_1} |S_{\theta_j} v_j|_{s+m} |S_{\theta_j} v_j|_{s_0+m},$$
and, bounding $s_0 + m$ norms by $s+m$ norms, and using $|u_j|_{s+m} + |v_j|_{s+m} \lesssim \theta_j^{-\a},$ a consequence of $({\cal C}_1(j')_{j'\leq j},$ and \eqref{con1}, we obtain
\begin{equation} \label{bd-E2}
 E_2 \lesssim \e^{-\g_1} \theta_j^{-2\a}.
 \end{equation}
  Bounds \eqref{bd-E1}, \eqref{bd-E2} and \eqref{con2} imply ${\cal C}_2(j+1)$(ii). Finally, to prove ${\cal C}_2(j+1)$(iii), we remark that, by \eqref{4},  
  \begin{eqnarray} |u_{j+1}|_{s + p} & \leq &  |u_{j}|_{s + p} +  |S_{\theta_j} v_{j}|_{s + p}\nonumber \\
 & \lesssim &   |u_{j}|_{s + p} +  \theta_j^{\max(m+r',r)} |  v_{j}|_{s + p - \max(m+r',r)}.\label{a9} \end{eqnarray}
 Under \eqref{con-p}, the tame direct bound \eqref{1} in Assumption \ref{ass1} and the tame inverse bound \eqref{2} in Assumption \ref{ass2} imply
  \begin{eqnarray} |  v_{j}|_{s + p - \max(m+r',r)} & \lesssim & \e^{-\k}(  |u_{j}|_{s + p} \| \Phi(u_j)\|_{s_0 + m + r'} + \| \Phi(u_j)\|_{s + p - \max(m+r',r) + r'} ) \nonumber\\
 & \lesssim &  \e^{-\k} (1 + |u_{j}|_{s + p}) (1 + \| \Phi(u_j)\|_{s_0 + m + r'}) \nonumber \\
 & \lesssim & \e^{-\k} \theta_j^N. \label{a10} 
 \end{eqnarray}
 Bounds \eqref{a9}, \eqref{a10} 
 and \eqref{con3} imply ${\cal C}_2(j+1)$(iii). 
 \end{proof}

\begin{lem} \label{l3} Let $j$ such that 
\begin{equation} \label{con4}
 \e^{-\k} \theta_{j}^{-\b} \leq \theta_{j}^{-\a}
 \end{equation}
 where $$\b := (p'+\max(r,r'))^{-1} ((p'-q)- N (q + \max(r,r'))), \qquad p' := p - \max(m+r',r).$$
 Then condition ${\cal C}_2(j)$
 implies ${\cal C}_1(j).$
\end{lem}

\begin{proof} Bound ${\cal C}_2(j)$(i), together with \eqref{cond1-0}, implies that $u_{j}$ satisfies \eqref{cond1}. Then, bound ${\cal C}_2(j)$(iii) implies that $v_{j}$ is well defined in $E_{s+p-\max(m+r',r)},$ and we can check, exactly as in the proof of \eqref{a10} in Lemma \ref{l2}, that the bound 
 \begin{equation} \label{new}
  | v_{j}|_{s + p - \max(m+r',r)} \lesssim \e^{-\k} \theta_{j}^N
  \end{equation}
 holds. Besides, by the tame inverse bound \eqref{2} in Assumption \ref{ass2},
 \begin{eqnarray} | v_{j}|_{s - \max(r,r')} & \lesssim & \e^{-\k}( | u_{j}|_{s} \| \Phi(u_{j})\|_{s_0 + m} + \| \Phi(u_{j})\|_{s}) \nonumber  \\ 
 & \lesssim & \e^{-\k} (1 + \theta_0^{-\a}) \| \Phi(u_{j})\|_{s} \nonumber \\ & \lesssim & \e^{-\k} \theta_{j}^{-1}. \label{v} \end{eqnarray}
Finally, bounds \eqref{new}, \eqref{v} and the interpolation property \eqref{sp2} imply
\begin{eqnarray} | v_{j} |_{s + q}  & \lesssim  & \nonumber |v_{j}|_{s-r''}^{\frac{p'-q}{p' +r''}} |v_{j}|_{s + p'}^{\frac{q+r''}{p'+r''}} \\ & \lesssim & \e^{-1} \theta_{j}^{-\b},\end{eqnarray}
 where $r'' = \max(r,r'),$ and the Lemma follows, with \eqref{con4}. 
   
 \end{proof}

\begin{proof}[End of proof of Theorem {\rm \ref{th1}}, existence] Let $q = m + \a.$ Define 
 \begin{equation} \label{def-theta} \theta_0 := \e^{-k}, \qquad \theta_{j+1} := \theta_j^\zeta, \quad j \geq 0,\end{equation}
 for some $\zeta > 1$ to be chosen below. Then \eqref{cond1-0} is satisfied if
  \begin{equation} \label{l-k}
    \max\left(\frac{\gamma_0}{\a},\frac{\gamma}{\a}\right) \leq k,
   \end{equation}
  and \eqref{con1} is satisfied. 

 By \eqref{def-theta} and Assumption \ref{ass-wkb}, condition ${\cal C}_2(0)$ is satisfied. Condition \eqref{con4} is satisfied as soon as
 \begin{equation} \label{beta'}
 \frac{\k}{\b- \a} \leq k.
 \end{equation}
 By Lemma \ref{l3}, \eqref{beta'} also implies that condition ${\cal C}_1(0)$ is satisfied. With definition \eqref{def-theta}, conditions \eqref{con2} and \eqref{con3} translate respectively into 
 \begin{equation} \label{z0}
  \frac{\g_1}{k} < 2 \a - \zeta,
 \quad \mbox{ and } \quad  \frac{\k}{N \zeta - N - \max(m+r',r)} \leq k.
  \end{equation}

 Suppose now that for all $0 \leq j' \leq j,$ $u_{j'}$ is well defined and ${\cal C}_1(j')$ and ${\cal C}_2(j')$ hold. Then by Lemma \ref{l2}, condition ${\cal C}_1(j+1)$ is satisfied if \eqref{z0} holds, and by Lemma \ref{l3}, condition ${\cal C}_2(j+1)$ is satisfied if \eqref{beta'} holds.
 
 We just proved that, under \eqref{l-k}, \eqref{beta'} and \eqref{z0}, conditions ${\cal C}_1(j)$ and ${\cal C}_2(j)$ hold for all $j.$ 
 
 Conditions \eqref{l-k}, \eqref{beta'} and \eqref{z0} are equivalent to 
 \begin{equation} \label{ens-cond}
   M \leq \frac{1}{2} \left(\zeta + \frac{\g_1}{k}\right) < \a \leq \left(1 + \frac{1}{p_0}\right)^{-1} \left(1 - \frac{\k}{k} - \frac{1}{p_0}( m + r'')\right).
  \end{equation}
 with 
 $$ M := \max\left(\frac{\g_0}{k}, \frac{\g}{k}, \frac{1}{2} \left(1 + \frac{\g_1}{k} + \frac{m'}{N} + \frac{\k}{k N} \right)\right), \quad p_0:= \frac{p - m' + \max(r,r')}{N + 1},$$
 and  $m' := \max(m+r',r).$ Under \eqref{ass-k}, if $N$ and $p$ are large enough, namely 
 \begin{equation} \label{N0}
  \frac{\k + k m'}{k - (2\k + \g_1)} =: N_0 < N, \qquad \bar p < p,
 \end{equation}
 where $\bar p$ is specified in Remark \ref{rem3}, then we can find $\a$ and $\zeta$ satisfying \eqref{ens-cond}.

 Let now $\a,$ $\zeta,$ $N,$ and $p$ such that \eqref{ens-cond} holds.
 By \eqref{def-theta} and $\zeta > 1,$ the series $\theta_j^{-\a}$ is convergent and satisfies \eqref{cv-series}. Besides, conditions ${\cal C}_2(j)$ imply $\| \Phi(u_j)\|_s \to 0.$ We can thus apply Lemma \ref{l1}: the sequence $u_j$ converges to a solution $u$ of \eqref{0} in $s+q$ norm, satisfying \eqref{e1}. Besides, as \eqref{v} holds for all $j,$
$$ |u_j |_s  \lesssim \e^{-\k} \sum_{j'=0}^j \theta_{j'}^{-\b} \lesssim \e^{-\k} \theta_0^{-1},
$$
 hence \eqref{est-th}.

\end{proof}

\begin{proof}[Proof of Theorem {\rm \ref{th2}}, local uniqueness.]
Suppressing $\eps$, let $\hat u$ be a second solution in $E_{s_0 + 2 m + r'}$ of $\Phi(u)=0$, lying
within $o(\eps^{\max(\k + \g_1,\gamma_0,\gamma)})$ of $u$ (and
thus of $0$).
Then, Taylor expanding, and using Assumption \ref{ass1}, we have
$$
0=\Phi(\hat u)-\Phi(u)=
\Phi'(u)(\hat u-u) + B(u, \hat u),
$$
 where 
 $$ B(u, \hat u) := \int_0^1 (1 - t)\Phi''(t u + (1 - t ) \hat u) \cdot (\hat u-u,\hat u-u) \, dt.$$
Applying $\Psi(u)$ and using Assumption \ref{ass2}, we thus have
$$
(\hat u-u) +
\Psi(u) B(u, \hat u) \in 
\kernel \Phi'(u),
$$
where
$$ \begin{aligned}
|\Psi(u) B(u, \hat u)|_{s_0} & \lesssim
\eps^{-\k - \g_1} \big(1 + \e^{-\g_1} (|\hat u|_{s_0 + m} + |u|_{s_0+m})\big)|\hat u - u|_{s_0 + 2 m + r'}^2  \\ & 
 = o(|\hat u-u|_{s_0 + 2 m + r'}).\end{aligned}
$$
This verifies tangency.
 Finally, from $\hat u-u + o(|\hat u-u|)\in \kernel \,(\Phi^\e)'(u^\e)$, we have 
 $$
\hat u-u + o(|\hat u-u|)=
\Pi_{u^\e} (\hat u^\e-u^\e)+
o(| \Pi_{u^\e}| |\hat u-u|),
$$
which, with \eqref{phasecond} and the assumed uniform boundedness of
$| \Pi_{u^\e}|$, gives
 $$
\hat u-u =
 o(|\hat u-u|)+ o(| \Pi_{u^\e}| |\hat u-u|)=
 o(|\hat u-u|),
$$
and thus $\hat u-u =0$.
\end{proof}

\end{section}

\section{Application: systems of quasilinear Schr\"odinger equations} \label{sec:app1}

 Consider systems of quasilinear Schr\"odinger equations in $v = (v_1,\dots,v_n) \in \C^n,$ 
 \begin{equation} \label{qlS}
  \d_t v_j + i \l_j \Delta_x v_j = \sum_{1 \leq j' \leq n} b_{jj'}(v,\d_x) v_{j'} + c_{jj'}(v,\d_x) \bar v_{j'}, \qquad 1 \leq j \leq n, \quad t \geq 0, x \in \R^d, \end{equation}
 with $d \geq 2.$ The $\l_j$ are assumed to be real and pairwise distinct, and the coefficients $b_{jj'}$ and $c_{jj'}$ are first-order differential operators:
  $ (b_{jj'},c_{jj'})(v,\d_x) = \sum_{1 \leq k \leq d} (b_{kjj'}(v),c_{kjj'}(v)) \d_{x_k},$
  where the maps $v \in \C \to (b_{kjj'}, c_{kjj'})(v) \in \C^2$ are smooth and satisfy, for some $\ell \in \N$ with $\ell \geq 2,$ and some $C > 0,$ for all $0 \leq |\a| \leq 2,$ for all $v,$
  \begin{equation} \label{homog}
  |\d_v^\a b_{kjj'}(v)| + | \d_v^\a c_{kjj'}(v)| \leq C |v|^{\ell - |\a|}.
  \end{equation}

We make the following assumption:
\begin{ass} \label{ass:S2} For all $j,j'$ such that $\l_j + \l_{j'} = 0,$ there holds $c_{jj'} = c_{j'j}.$ For all $j,$ there holds $\Im m \, b_{jj} \equiv 0.$
\end{ass}

  Assumption \ref{ass:S2} is a ``transparency" condition, similar to Assumptions 2.1, 2.5 and 2.10 in {\rm\cite{JMR}} and Assumption 2.15 in {\rm \cite{T1}}. It means that the singular source terms in \eqref{qlS2} possess some favorable structure (cancellation or symmetry) at the resonances.

Consider a family of initial data
 \begin{equation} \label{init-qlS}
  v^\e(0,x) = \e^{\s} a_\e(x), \qquad \mbox{with \quad $\sup_{\e \in (0,1)} \| a_\e\|_{H^{s_1}_\e} < \infty$},
 \end{equation}
 where $\s > 0$ and $a_\e$ is for instance concentrating: $a_\e(x) = a^0\left(\frac{x}{\e}\right),$ or oscillating: $a_\e(x) = a^0(x) e^{i x \cdot \xi_0/\e},$ for some $\xi_0 \in \R^d;$ in both cases $a^0 \in H^{s_1},$ for some large $s_1.$

   Our goal is to show that, under Assumption \ref{ass:S2}, for $s_1$ and $\s$ large enough, any $T >0$ and $\e$ small enough, we can apply Theorem \ref{th1} to prove existence over $[0,T],$ in weighted Sobolev spaces, for the initial-value problem \eqref{qlS}-\eqref{init-qlS}.

\begin{ex} \label{ex:S} Our assumptions are satisfied in particular by systems 
 $$ \left\{ \begin{aligned} \d_t v_1 + i \Delta v_1 & = b_{12}(v,\d_x) v_2 +  c_{11}(v,\d_x) \bar v_1 + c(v) \d_x \bar v_2, \\ \d_t v_2 - i \Delta v_2 & = b_{22}(v,\d_x) v_2 + c(v) \d_x \bar v_1 + c_{22}(v,\d_x) \bar v_2,\end{aligned}\right.$$
if $b_{22}$ is real, $b_{12},$ $b_{22},$ $c_{11},$ and $c_{22}$ are first-order differential operators, and all coefficients are $\ell$-homogeneous in $v,$ for some integer $\ell \geq 2.$ 
\end{ex}

\begin{rem} \label{rem:GM} Rauch and M\'etivier give in {\rm \cite{MR}} (Theorem 1.5; see also {\rm \cite{Me}}, Theorem 8.1.2) a local existence and uniqueness result for the Cauchy problem for \eqref{qlS}, under Assumption {\rm \ref{ass:S2}}, for data in $H^s,$ with $s > 1 +d/2.$ There is no small parameter in their setting.  We compare Rauch and M\'etivier's result with ours in Remark {\rm \ref{rem:comparison}.}
\end{rem}

 \subsection{Semi-classical setting} \label{sec:semicl}
 
 Introducing $u = (v,\bar v)^T \in \C^{2n},$ we obtain a system
 \begin{equation} \label{sys}
  \d_t u + i A(\d_x) u = B(u,\d_x) u,
 \end{equation}
 where $A$ is the diagonal, second-order, constant-coefficient operator
 \begin{equation} \label{A}
  A(\d_x) = \mbox{diag}\left(\l_1, \dots, \l_n,-\l_1,\dots,-\l_n\right) \Delta_x,
  \end{equation}
  and $B$ is the first-order operator
  $$ %
  B  = \left(\begin{array}{cc} {\cal B} & {\cal C} \\ \bar{\cal C} & \bar{\cal B} \end{array}\right), \qquad {\cal B} := \left( b_{jj'} \right)_{1 \leq j,j' \leq n}, \quad {\cal C} := \left( c_{jj'} \right)_{1 \leq j,j'\leq n}.
  $$ %
  Let
 $$ %
 J := \big\{ (j,j'), \quad \l_j + \l_{j'} = 0 \big\},
 $$ %
 and $\chi \in C^\infty_c(\R^d,\R)$ be a frequency truncation, such that $0 \leq \chi \leq 1,$ $\chi \equiv 1$ for $|\xi| \leq 1/2$ and $\chi \equiv 0$ for $|\xi| \geq 1.$
The source $B$ in \eqref{sys} decomposes into the sum of a resonant, a non-resonant term, and a low-frequency term:
$ B = B_r + B_{nr} + B_{lf}, $
where 
\begin{itemize} \item the resonant term is 
 $$ B_r := \mbox{diag}\big(b_{11}, \dots, b_{nn}, \bar b_{11}, \dots \bar b_{nn}\big) + \left(\begin{array}{cc} 0 & {\cal C}_{J} \\ \bar {\cal C}_{J} & 0 \end{array}\right),$$
 with the notation 
 $({\cal C}_{J})_{jj'} := c_{jj'}$ if $(j,j') \in J,$ and $({\cal C}_{J})_{jj'} := 0$ otherwise.
 The key is that, under Assumption \ref{ass:S2}, for all $v,\xi,$ the matrix $B_r(v,\xi)$ is hermitian;  %
\item the non-resonant term is 
 $$ B_{nr} := \left(\begin{array}{cc} {\cal B}^1 & {\cal C}^1 \\ \bar{\cal C}^1 & \bar{\cal B}^1 \end{array}\right),$$
 where $\left({\cal B}^1\right)_{jj'} := (1 - \chi) b_{jj'}$ if $j \neq j',$ $\left({\cal B}^1\right)_{jj'} := 0$ otherwise; $({\cal C}^1)_{jj'} := (1 - \chi) c_{jj'}$ if $(j,j') \notin J,$ $\left({\cal C}^1\right)_{jj'} := 0$ otherwise; %
\item the low-frequency term is
 $$ B_{lf} :=  \left(\begin{array}{cc} {\cal B}^0 & {\cal C}^0 \\ \bar{\cal C}^0 & \bar{\cal B}^0 \end{array}\right),$$
  where $\left({\cal B}^0\right)_{jj'} := \chi b_{jj'}$ if $j \neq j',$ $\left({\cal B}^0\right)_{jj'} := 0$ otherwise; $({\cal C}^0)_{jj'} := \chi c_{jj'}$ if $(j,j') \notin J,$ $\left({\cal C}^0\right)_{jj'} := 0$ otherwise.
 \end{itemize}
  By assumption, $B$ is homogenenous degree one in $\xi.$ Taking into account the dependence of the datum in $x/\e,$ and using the homogeneity of $A$ and $B,$ we work with weighted derivatives, and rewrite \eqref{sys} as
 \begin{equation} \label{qlS2}
  \d_t u + \frac{i}{\e^2} A(\e\d_x) u = \frac{1}{\e} B(u,\e\d_x) u. 
  \end{equation}
 The family of initial-value problems \eqref{qlS2}-\eqref{init-qlS} corresponds to the equation $\Phi^\e(u) = 0$ for the family of maps
  \begin{equation} \label{phi-qlS}
   \Phi^\e(u) := \left(\begin{array}{l} \e^2 \d_t u + i A(\e \d_x) u - \e B(u,\e\d_x) u  \\  u_{|t = 0} - \e^\s a_\e \end{array}\right).
   \end{equation}

 Given $T > 0,$ consider the functional spaces
 \begin{equation} \label{esfs2} E_s =  H^{s}(\R^d), \quad  F_{s-2}  =  C^0([0, T], H^{s-2}(\R^d)) \times H^{s}(\R^d),
 \end{equation}
 with norms
 $$ | u |_{s} := \sup_{0 \leq t \leq T} ( \| \e^2 \d_t u(t) \|_{H^{s-2}_\e} + \| u(t) \|_{H^s_\e} ), \quad  \| (f_1,f_2) \|_{s-2}  := \sup_{0 \leq t \leq T} \| f_1(t)\|_{H^{s-2}_\e} + \| f_2\|_{H^{s}_\e},$$
 where the weighted Sobolev norms $\| \cdot \|_{H^s_\e}$ are defined in \eqref{hse}. By definition, $\Phi^\e$ belongs to $C^2(E_s,F_{s-2}),$ for all $s$ such that $s_0 + 2 \leq s \leq s_1,$ for any $d/2 < s_0 < 1 + d/2.$

\subsection{Tame direct bounds} \label{sec:direct}
 
 Given $a^0 \in H^{s_1},$ there holds $\sup_{\e \in (0,1)} \| a^0(x/\e) \|_{H^{s_1}_\e} < \infty,$ $\sup_{\e \in (0,1)} \| a^0(x) e^{i x \cdot \xi_0/\e} \|_{H^{s_1}_\e} < \infty.$ We assume that $s_0 + 4 \leq s_1.$
Let $s_0 + 2 \leq s \leq s_1 -2,$ and $u \in H^{s+2}.$ There holds 
 \begin{equation} \label{first-phi} \| \Phi^\e(u) \|_{s} \leq \e^\s C(a^0,s) + C(\l_j) |u|_{s+2} +\e \| B(u,\e\d_x) u \|_{H^s_\e},
 \end{equation}
for some $C(a^0,s) > 0$ and $C(\l_j) > 0.$

 \begin{lem} \label{lem:direct} The family $\Phi^\e$ defined in \eqref{phi-qlS} satisfies Assumption {\rm \ref{ass1}} with $(\g_0,\g_1) \in \R_+ \times \R_+$ such that
 \begin{equation} \label{cond-g-l}
   \begin{aligned}
   \left. \begin{aligned} 1 - \frac{d\ell}{2} + \g_0 \ell + \min\big(0, \g_1 - \g_0\big) & \geq 0 \\  1 - d + \min(\g_0 + \g_1, 2\g_1) & \geq 0 \end{aligned}\right\} & \quad \mbox{if $\ell = 2,$} \\ 1 - \frac{d\ell}{2} + \g_0 \ell + \min\big(0, \g_1 - \g_0, 2 (\g_1 - \g_0) \big) \geq 0, & \quad \mbox{if $\ell \geq 3.$} 
    \end{aligned}
    \end{equation}
 \end{lem}
 
 \begin{proof} We start from \eqref{first-phi} and bound the differential operator $B(u,\e\d_x) u$ as in Example \ref{ex3}. 
  By \eqref{homog}, for $s > d/2$ and $|u|_{L^\infty} \leq M_0,$ there holds $\| d(u) \|_{H^s_\e} \lesssim C(M_0) |u|_{L^\infty}^{\ell-1} \| u \|_{H^s_\e},$ with $d = b_{jj'}, c_{jj'}.$ Thus we obtain
 \begin{equation} \label{brnr} \e \| B (u,\e\d_x) u \|_{H^s_\e} \lesssim \e^{1 - d \ell/2} \| u \|_{H^{s_0+1}_\e}^{\ell}  \| u \|_{H^{s+1}_\e}.
 \end{equation}
 It follows that \eqref{1} holds for any $\g_0$ such that $1 - d \ell /2 + \g_0 \ell \geq 0.$ The first derivative $(\Phi^\e)'(u)$ involves $(\d_u B(u, \e\d_x) \cdot v) u + B(u,\e\d_x) v,$
where $(\d_u B(u, \e\d_x) \cdot v) u$ is a differential operator acting on $u,$ with coefficients depending on $v,$ and satisfies 
 $$ \begin{aligned} \e \| (\d_u B(u,\e\d_x) \cdot v ) u \|_{H^s_\e} \lesssim \e^{1 - d \ell/2} \| u \|_{H^{s_0+1}_\e}^{\ell -1} \big( \| v \|_{H^{s_0}_\e} \| u \|_{H^{s+1}_\e} +  \| u \|_{H^{s_0+1}_\e}  \| v \|_{H^s_\e} \big).
 \end{aligned}
 $$
The other term in the first derivative, $B(u,\e\d_x) v,$ is bounded as in \eqref{brnr}, and we obtain
$$ \| (\Phi^\e)'(u) v \|_s \lesssim | v |_{s+2} + \e^{1 - d\ell/2}  \| u \|_{H^{s_0+1}_\e}^{\ell -1} \big(  \| u \|_{H^{s_0+1}_\e} \| v \|_{H^{s+1}_\e} + \| v \|_{H^{s_0}_\e}  \| u \|_{H^{s+1}_\e}\big).$$
The bound for the second derivative is similar, and we obtain the bounds of Assumption \ref{ass1} under condition \eqref{cond-g-l}. 
 \end{proof}
\subsection{Tame inverse bounds}

 For the linearized system of quasilinear Schr\"odinger equations $(\Phi^\e)'(\un u) u = (f_1,f_2),$ explicitly
 \begin{equation} \label{lin-qlS2}
   \left\{\begin{aligned}
    \d_t u + \frac{i}{\e^2} A(\e\d_x) u & = \frac{1}{\e} B(\un u, \e \d_x) u + (\d_u B(\un u,\d_x) \cdot u) \un u + \frac{1}{\e^2} f_1, \\
     u_{|t  = 0 } & = f_2,
      \end{aligned}\right.
  \end{equation}
 we give a tame bound for $u,$ of the form \eqref{2}, by 
 using the ``transparency" condition expressed in Assumption \ref{ass:S2}. 
 The key is that, by Assumption \ref{ass:S2}, the matrix $B_r(\un u,\xi)$ is hermitian for all $(\un u,\xi),$ while $B_{nr}(\un u, \xi)$ corresponds to non-resonant interactions and can be eliminated by a normal form reduction. The other singular term, $(1/\e) B_{lf},$ is a low-frequency term, hence its singular prefactor does not harm the estimate. 
  The non-singular term $\d_u B(\un u, \d_x) \cdot u) \un u$ is a differential operator acting on $\un u;$ we denote it $D := D(u, \un u, \d_x \un u).$

 In the proof of Lemma \ref{lem:nf}, we use the notation and results of Section \ref{sec:pdo} on pseudo-differential symbols and operators. 
  
 \begin{lem}\label{lem:nf} Given $T> 0,$ $s_0 + 2 \leq s \leq s_1 - 2,$ $f \in F_{s+2}$ and $\un u \in E_{s+2},$ if $\un u$ satisfies
 \begin{equation} \label{un-u}
   |\un u|_{s_0 + 4} \lesssim \e^{\g},  \qquad \g = \frac{d}{2} + \frac{d}{2(\ell - 1)}.
   \end{equation}
  then there exists a unique $u \in E_s$ satisfying \eqref{lin-qlS2}, and there holds
   \begin{equation} \label{2app}| u |_{H^s_\e}  \lesssim \e^{-2} \| f \|_{s} + \e^{-3}  \| f \|_{s_0 + 2} | \un u |_{s+2}.
   \end{equation}
  \end{lem}

   \begin{proof}
 Our goal is to prove estimates over $[0,T]$ for \eqref{lin-qlS2}; existence and uniqueness then follow by classical arguments. We do not expect the estimates to be uniform in $\e,$ and aim for polynomials prefactors in $\e^{-\k} e^{C t},$ 
  for some $C>0;$ 
 in this view, the only obstacle is the singular term $(1/\e) B(\un u, \e \d_x)$ in the right-hand side, which, by direct bounds and Gronwall's Lemma, a priori contributes $e^{C t/\e}.$
 
 We look for a pseudo-differential matrix symbol $M = M(\un u, \xi) = (M_{jj'}(\un u,\xi))_{1 \leq j,j'\leq 2n}$ that belongs to the class $\G^{-1}_s$ defined below in Section \ref{sec:pdo}, such that, using the notation \eqref{semicl} for pseudo-differential operators in semi-classical quantization,
the map
  \begin{equation} \label{checku}
  \check u := (\Id + \e \op_\e(M))^{-1} u
  \end{equation}
 satisfies an equation that will allow an estimate of the form \eqref{2app}. If $u$ solves \eqref{lin-qlS2}, then $\check u$ solves
 $ \d_t \check u  = {\cal A} \check u  + g,$
 where
 $$ {\cal A} := (\Id + \e \op_\e(M))^{-1} \Big( - \frac{i}{\e^2} A(\e\d_x) + \frac{1}{\e} B(\un u, \e\d_x) + D \Big)(\Id + \e \op_\e(M)),$$
 and $ g := (\Id + \e \op_\e(M))^{-1}\big(\e^{-2} f_1 - \e \op_\e(\d_t M) \check u \big).$
By Lemma \ref{composition},
 $$ {\cal A} = - \frac{i}{\e^2}  A(\e\d_x) + \frac{1}{\e} (B_r + B_{lf})(\un u, \e \d_x) u + \frac{1}{\e} \op_\e(H) + D +  \op_\e(E),$$
 where  
 $$ H(t,x,\xi) := B_{nr}(\un u(t,x), i \xi) - i \big[ A(i \xi), M(\un u(t,x), \xi)\big].
 $$ %
  and the remainder $E$ is
$$ \begin{aligned}
 \op_\e( E ) & :=  \op_\e(\tilde M)\big(-  i A(\e\d_x) + \e (B_r + B_{lf})(\un u, \e \d_x) + \e^2 D \big)(\Id + \e \op_\e(M)) - R(A,M) \\
   & \quad  + \op_\e(M) \big(-  i A(\e \d_x) + \e (B_r + B_{lf})(\un u, \e\d_x) + \e^2 D \big) \op_\e(M),
   \end{aligned}
$$  %
 with $\e^2 \op_\e(\tilde M) := (\Id + \e \op_\e(M))^{-1} - \Id + \e \op_\e(M).$
 We used above the notation $R$ for remainders introduced in Lemma \ref{composition}. 
 
 By the diagonal structure of $A,$ the matrix commutator $[A(i \xi),M]$ is
 $$ [A(i\xi),M] = \left( (\L_j - \L_{j'}) M_{jj'}\right)_{1 \leq j,j'\leq 2n},$$
 where 
 $$\L_j = - \l_j |\xi|^2, \quad \mbox{if $1 \leq j \leq n,$}  \qquad \L_j = \l_{j-n} |\xi|^2, \quad \mbox{if $n+1 \leq j \leq 2n,$}$$
  in accordance with \eqref{A}. We note that, since the $\l_j$ are pairwise distinct, $ \L_j - \L_{j'} = 0$ if and only if $1 \leq j \leq n$ and $n+1 \leq j' \leq 2n$ with $(j,j'-n) \in J,$ or $n+1 \leq j \leq 2n$ and $1 \leq j'\leq n,$ with $(j-n,j') \in J.$ By definition of $B_{nr}$ in Section \ref{sec:semicl}, for such couples $(j,j'-n)$ and $(j-n,j'),$ there holds $(B_{nr})_{jj'} \equiv 0.$ Besides, by definition of $B_{nr},$ for small $\xi$ there holds $B_{nr} \equiv 0.$ This implies that 
  $$ M_{jj'}(\un u(t,x),\xi) := \left\{ \begin{aligned}  - i \left( \L_{j}(\xi) - \L_{j'}(\xi) \right)^{-1} (B_{nr})_{jj'}(\un u(t,x), i \xi), & \qquad \mbox{if $\L_j - \L_{j'} \neq 0$}, \\ 0, & \qquad \mbox{if $\L_j - \L_{j'} = 0.$}\end{aligned}\right.
  $$ %
 defines a symbol $M \in \G^{-1}_s.$ With this choice of $M,$ there holds $H \equiv 0,$ and the equation in $\check u$ simplifies into
 \begin{equation} \label{reduced}
  \d_t \check u + \frac{i}{\e^2} A(\e \d_x) \check u = \frac{1}{\e} (B_r + B_{lf})(\un u, \e\d_x) \check u + D + \op_\e(E) \check u  + g.
  \end{equation}
 We now perform direct estimates on the reduced equation \eqref{reduced}. By reality of the $\l_j,$
 $$ \Re e \, \frac{i}{\e^2} \big( A(\e \d_x) \check u, \check u\big)_{H^s_\e} = 0.$$
 By the hermitian structure of $B_r(\un u,\xi)$ and Lemma \ref{adjoint},
 $$ \begin{aligned} \Re e \, \frac{1}{\e} \left( B_r(\un u, \e \d_x) \check u, \check u\right)_{H^s_\e} & \lesssim | \un u |_{L^\infty}^{\ell - 1} | \un u |_{W^{1,\infty}} \| \check u \|_{H^s_\e}^2 + \e^{-1 - d/2} | \un u |_{L^\infty}^{\ell - 1} \| \check u \|_{H^{s_0+1}_\e} \| \un u \|_{H^s_\e}  \| \check u \|_{H^s_\e}.
 \end{aligned}
 $$ %
 By Lemma \ref{action} with $m = s_0-s,$ under \eqref{un-u},
 $$ \begin{aligned} 
 \frac{1}{\e} \| B_{lf}(\un u, \e\d_x) \check u \|_{H^s_\e} &  \lesssim \e^{-1} \| \check u \|_{H^{s_0}_\e} ( |\un u |_{L^\infty}^\ell + \e^{-d/2} | \un u |_{L^\infty}^{\ell - 1} \| \un u \|_{H^s_\e})  
 \end{aligned} 
 $$ %
 The zeroth-order term $D$ satisfies
 $$  \begin{aligned} \| D \|_{H^s_\e} & \lesssim |\un u |_{L^\infty}^{\ell - 1} |\d_x \un u |_{L^\infty} \| \check u \|_{H^s_\e} + \e^{-d/2} |\un u |_{L^\infty}^{\ell - 2} | \un u |_{W^{1,\infty}} \| \check u \|_{H^{s_0}_\e} \| \un u \|_{H^{s+1}_\e}. 
 \end{aligned}
 $$ %
 The change of variable $M$ satisfies, for all $w \in H^{s-1},$ 
 $$ \begin{aligned} \| \op_\e(M) w \|_{H^s_\e} & \lesssim C(|\un u |_{L^\infty}) \big(| \un u |^\ell_{L^\infty} \| w \|_{H^{s-1}_\e} + \e^{-d/2} | \un u |_{L^\infty}^{\ell-1}  \| w \|_{H^{s_0}_\e} \| \un u \|_{H^s_\e}\big). 
  \end{aligned}
 $$  %
 Let us now restrict to a background $\un u$ satisfying \eqref{un-u}. Then, the above bounds become
 $$ \begin{aligned}
  \Re e \, \frac{1}{\e} \left( B_r(\un u, \e \d_x) \check u, \check u\right)_{H^s_\e} & \lesssim  \big( \e^{-1 + \g} \| \check u \|_{H^s_\e} + \e^{-1} \| \check u \|_{H^{s_0+1}_\e} \| \un u \|_{H^s_\e} \big) \| \check u \|_{H^s_\e}\\
  \frac{1}{\e} \| B_{lf}(\un u, \e\d_x) \check u \|_{H^s_\e} &  \lesssim  \big(\e^{-1 + \g} \| \check u \|_{H^{s_0}_\e} + \e^{-1} \| \check u \|_{H^{s_0}_\e} \| \un u \|_{H^s_\e}\big)  \| \check u \|_{H^s_\e} \\ 
  \| D \|_{H^s_\e} & \lesssim \e^{-1 + \g} \| \check u \|_{H^{s}_\e} + \e^{-1} \| \check u \|_{H^{s_0}_\e} \| \un u \|_{H^{s+1}_\e} \\ 
  \| \op_\e(M) w \|_{H^s_\e} & \lesssim \e^{\g} \| w \|_{H^{s-1}_\e} +\| w \|_{H^{s_0}_\e} \| \un u \|_{H^s_\e},
  \end{aligned}$$
 for all $w \in H^{s-1}.$  %
 In particular, for all $w \in H^{s_0},$ 
 \begin{equation} \label{Ms0}
 \| \op_\e(M) w \|_{H^{s_0}_\e} \lesssim \e^{\g} \| w \|_{H^{s_0}_\e}.
 \end{equation}
 A consequence of \eqref{Ms0} is that, given $k \geq 2,$ the operator $\op_\e(M)^k$ maps $H^s$ to $H^{\max(s-k,s_0)},$ and for all $w \in H^{\max(s-k,s_0)},$
 $$ \begin{aligned} \| \op_\e(M)^k w \|_{H^s} & \lesssim {\bf M}^{-1}_0(M)^k \| w \|_{H^{s-k}_\e} + \e^{-d/2} \sum_{0\leq k' \leq k-1} {\bf M}^{-1}_0(M)^{k'} N_{s-k'}^{-1}(M) \| w \|_{H^{s_0}_\e} \\ & \lesssim  |\un u|_{L^\infty}^{\ell k} \| w \|_{H^{s-k}_\e} + \e^{-d/2} \sum_{0 \leq k' \leq k-1} |\un u|_{L^\infty}^{\ell(k' + 1) - 1} \| \un u \|_{H^{s-k'}_\e} \| w \|_{H^{s_0}_\e} \\
  & \lesssim \e^{\g k} \| w \|_{H^{s-k}_\e} +  k \| \un u \|_{H^{s}_\e} \| w \|_{H^{s_0}_\e}. \end{aligned}$$%
 It follows that $\op_\e(\tilde M) = \sum_{k \geq 2} (-\e)^{k-2} \op_\e(M)^k$ maps $H^s$ to $H^{s-2},$ and for all $w \in H^{s-2},$ 
 $$ \| \op_\e(\tilde M) w \|_{H^s_\e} \lesssim \e^{2(\g - 1)} \| w \|_{H^{s-2}_\e} + \| \un u \|_{H^s_\e} \| w \|_{H^{s_0}}.$$
 The above bounds and Lemma \ref{composition} imply
 $$ \| \op_\e(E) \check u \|_{H^s_\e} \lesssim \| \check u \|_{H^s_\e} + \e^{-1} \| \check u \|_{H^{s_0+1}_\e} \| \un u \|_{H^{s+2}_\e} .$$
 The remainder $g$ satisfies
  $$ \begin{aligned} \| g \|_{H^s_\e} & \lesssim \e^{-2} \| f_1 \|_{H^s_\e} + \e^{-2} \| f_1 \|_{H^{s_0}_\e} \| \un u \|_{H^s_\e} +  \e |\un u |_{L^\infty}^{\ell - 1} |\d_t \un u |_{L^\infty} \| \check u \|_{H^{s-1}_\e} \\ & \quad + \e^{1 - d/2} |\un u |_{L^\infty}^{\ell - 2} \| \check u \|_{H^{s_0+1}_\e} \big( |\un u |_{L^\infty} \| \d_t \un u \|_{H^s_\e} + |\d_t \un u |_{L^\infty} \| \un u \|_{H^s_\e}\big) \\ & \lesssim \e^{-2} \| f_1 \|_{H^s_\e} + \e^{-2} \| f_1 \|_{H^{s_0}_\e} \| \un u \|_{H^s_\e} +  \e^{-1 + \g}  \| \check u \|_{H^{s-1}_\e} \\ & \quad + \e^{-1 + \g} \| \check u \|_{H^{s_0+1}_\e} |\un u |_{s+2}.\end{aligned}$$
 Collecting the above bounds, we obtain the estimate
 \begin{equation} \label{est-checku} \begin{aligned}
  \d_t \| \check u \|_{H^s_\e}^2 \lesssim   \| \check u \|_{H^s_\e}^2 + \Big( \e^{-1 + \g} \| \check u \|_{H^{s_0+1}_\e}  & + \e^{-1}  \| \check u \|_{H^{s_0+ 1}_\e}  \| \un u \|_{H^{s+2}_\e} \\ & + \e^{-2} \| f_1 \|_{H^s_\e} + \e^{-2} \| f_1 \|_{H^{s_0}_\e} \| \un u \|_{H^s_\e} \Big) \| \check u \|_{H^s_\e}, 
   \end{aligned}
 \end{equation}
 valid for any $s_0 + 2 \leq s \leq s_1 -2.$ We now let $s = s_0 + 2$ in \eqref{est-checku}, and obtain
 $$\| \check u \|_{H^{s_0+2}_\e} \lesssim \| f_2 \|_{H^{s_0 + 2 }_\e} + \e^{-2} \| f_1 \|_{H^{s_0+2}_\e} ,
 $$ %
 which we plug back in \eqref{est-checku} to get
 $$ \| \check u \|_{H^s_\e}  \lesssim \| f_2 \|_{H^{s}_\e}  + \e^{-2}  \| f_1 \|_{H^{s}_\e} + \e^{-1} \big(  \| f_2 \|_{H^{s_0 + 2}_\e}  + \e^{-2}  \| f_1 \|_{H^{s_0 + 2}_\e} \big) \| \un u \|_{H^{s+2}_\e}.
 $$ %
 In order to estimate $ \e^2 \d_t \check u,$ we use \eqref{lin-qlS2} directly, and, via \eqref{checku} and the estimate for the operator norm of $\op_\e(M),$ we finally obtain \eqref{2app}.
\end{proof}

\subsection{Result}

 Introduce $\s_a \geq 0,$ such that $\| a_\e \|_{H^{s_1}_\e} = O(\e^{\s_a}).$ For instance, in the concentrating case: $a_\e(x) = a^0(x/\e),$ we have $\s_a = d/2,$ and in the oscillating case: $a_\e(x) = a^0(x) e^{i x \cdot \xi_0/\e},$ we have $\s_a = 0.$ Introduce also the critical index $k_c$ defined by
  $$ k_c = \max(\k + \g_0, 2 \k + \g_1, \k + \g),$$
  where $\g_0, \g_1$ are given by Lemma \ref{lem:direct}, and $\g, \k$ by Lemma \ref{lem:nf}, so that $k_c$ depends only on $d$ and $\ell.$

 \begin{theo} \label{th-qlS} Under Assumption {\rm \ref{ass:S2}}, if the initial datum \eqref{init-qlS} is small enough and smooth enough, meaning that $s_1$ satisfies \eqref{low-s1} and
  \begin{equation} \label{M} 1 + \s(\ell + 1) + \s_a > k_c,
  \end{equation} 
  then, for any $T > 0,$ if $\e$ is small enough, the initial-value problem \eqref{qlS}-\eqref{init-qlS} has a solution $v \in C^1([0, T], H^{s-2}(\R^d)) \cap C^0([0, T], H^{s}(\R^d)),$ for $s$ satisfying \eqref{ass-s}.
 \end{theo}
 
 The regularity condition on the datum, \eqref{low-s1}, is meant here with $m= r = 2,$ $r' = 0,$ $\g_0, \g_1$ given by Lemma \ref{lem:direct}, and $\g, \k$ given by Lemma \ref{lem:nf}.
 
 \begin{proof} Let $d/2 < s_0 < 1 + d/2.$ The map $\Phi^\e$ defined in \eqref{phi-qlS} belongs to $C^2(E_s, F_{s-2}),$ where the functional spaces are defined in \eqref{esfs2}, for $s$ such that $s_0 + 2 \leq s \leq s_1,$ where $s_1$ is the assumed regularity of the datum $a^0.$ 
 
  For the values of the parameters given just above, we saw in Section \ref{sec:direct} that $\Phi^\e$ satisfies Assumption \ref{ass1}; besides, Lemma \ref{lem:nf} states that Assumption \ref{ass2} holds. 
  
  Let $a_f$ be the solution to the free Schr\"odinger system, and $\tilde \Phi^\e$ the family of shifted maps:
  $$ a_f(t,x) := \e^\s \exp\left(- i \frac{t}{\e^2} A(\e\d_x)\right) a_\e, \qquad \tilde \Phi^\e := \Phi^\e(a_f + \cdot).$$
  The family $\tilde \Phi^\e$ satisfies Assumption \ref{ass1} and \ref{ass2}, with the same parameters as $\Phi^\e,$ by Remarks \ref{rem:ass1-shift} and \ref{rem:ass2-shift}. There holds $ \tilde \Phi^\e(0) := \left( - \e B(a_f, \e\d_x) a_f, 0 \right),$
 so that 
 $$ \begin{aligned} \| \tilde \Phi^\e(0) \|_{s} & \lesssim \e |a_f|_{L^\infty}^{\ell - 1} \big( |a_f|_{L^\infty} \| a_f \|_{H^{s+1}_\e} + |\e\d_x a_f |_{L^\infty} \| a_f \|_{H^s_\e} \big) \\ & \lesssim \e^{1 + \s (\ell + 1)} \| a_\e \|_{H^{s+1}_\e}. \end{aligned}$$
 Condition \eqref{ass-k} here takes the form \eqref{M}. Under this condition, $\tilde \Phi^\e$ also satisfies Assumption \ref{ass-wkb}, and we conclude by application of Theorem \ref{th1}.
 \end{proof}
  
\subsection{Discussion and examples}

 Condition \eqref{M} relates the size of the datum in $L^\infty$ and $H^{s_1}_\e$ to the space dimension and the homogeneity of the differential operators in the system of quasilinear Schr\"odinger equations \eqref{qlS}.
  
  The following Remark explains how Theorem \ref{th-qlS} extends Rauch and M\'etivier's result mentioned in Remark \ref{rem:GM}:

\begin{rem} \label{rem:comparison} In Rauch and M\'etivier's result, Theorem 1.5 of {\rm \cite{MR}}, or Theorem 8.1.2 of {\rm \cite{Me}}, the existence time $T^*_s$ is a decreasing function of the initial Sobolev norm $\| \e^\s a_\e\|_{H^s};$ there holds $T^*_s \to 0$ as $\e \to 0$ if the datum tends to $+\infty$ in $H^s$ norm as $\e \to 0,$ and $T^*_s \to +\infty$ as $\e \to 0$ if the datum tends to $0$ in $H^s$ norm; besides, $T^*_{s'} \geq T^*_s$ if $s' \geq s.$ This features are shared with first-order quasilinear symmetric systems. 
 
 The datum in \eqref{init-qlS} satisfies in the concentrating case $\| a^0(x/\e) \|_{H^{1 + d/2}} = O(\e^{-1}),$ and in the oscillating case $\| a^0(x) e^{i x \cdot \xi_0/\e} \|_{H^{1+d/2}} = O(\e^{-1 - d/2}).$ Let $\s_1 = 1$ in the concentrating case and $\s_1 = 1 + d/2$ in the oscillating case, so that $\| \e^\s a_\e \|_{H^{1 + d/2}} = O(\e^{\s - \s_1})$ in both cases. Let $s$ and $s_1$ as in Theorem {\rm \ref{th-qlS},} and assume that $a^0 \in H^{s_1}.$
 
  Given any $s_0 > d/2,$ if $\s < \s_1,$ the datum $\e^\s a_\e$ is large in $H^{1 + s_0},$ hence large in $H^{s_1};$ in particular, $T^*_{s_1} \to 0$ as $\e \to 0.$ Assume now that in addition to $\s < \s_1,$ condition \eqref{M} holds. Then the datum is small in $H^{1 + s_0}_\e,$ but the equation written in $\e \d_x$ derivatives, \eqref{qlS2}, has large source terms. Under Assumption {\rm \ref{ass:S2}}, these terms are not present in the normal form of the equation, and Theorem {\rm \ref{th-qlS}} grants an arbitrarily long existence time in $H^{s}_\e.$ Thus, Theorem {\rm \ref{th-qlS}} extends Theorem 1.5 of {\rm \cite{MR}} (Theorem 8.1.2 of {\rm \cite{Me}}) in the case that both $\s < \s_1$ and condition \eqref{M} hold, for very regular data (indeed, $s_1 \gg 1 + d/2$ in practice, see Remark {\rm \ref{rem3}}).
\end{rem}

 For concentrating or oscillating data, the values of $\ell$ and $\s$ that allow both conditions $\s < \s_1$ and \eqref{M} to hold are described in the following:

\begin{ex} In two space dimensions, $d = 2:$
 \begin{itemize}
 \item in the concentrating case, conditions $\s < \s_1$ and \eqref{M} are incompatible if $\ell = 2$ and $\ell = 3,$ and they hold for $\frac{9}{2(\ell + 1)} < \s < 1$ if $\ell \geq 4;$
 \item in the oscillating case, conditions $\s < \s_1$ and \eqref{M} are incompatible if $\ell = 2;$ they hold for $\frac{5}{\ell + 1} < \s < 2$ if $\ell \geq 3.$  
 \end{itemize}
\end{ex}

\begin{ex} In three space dimensions, $d = 3:$
 \begin{itemize}
 \item in the concentrating case, conditions $\s < \s_1$ and \eqref{M} are incompatible if $\ell = 2$ and $\ell = 3,$ and they hold for $\frac{4}{\ell + 1} < \s < 1$ if $\ell \geq 4;$
 \item in the oscillating case, conditions $\s < \s_1$ and \eqref{M} hold for $2 < \s < 5/2$ if $\ell =2;$ they hold for $\frac{11}{2(\ell + 1)} < \s < 5/2$ if $\ell \geq  3.$
 \end{itemize}
 \end{ex}

 \subsection{Pseudo-differential symbols and operators} \label{sec:pdo}
 
    Given $m, s \in \R,$ we define the class $\Gamma^m_s$ as the space of symbols $\s$ defined on $\R^d_x \times \R^d_\xi,$ such that, for all $k \in \N,$ $\s \in C^{k}(\R^d_\xi;H^s(\R^d_x)),$ and
 $$ {\bf N}^m_{k,s}(\s) := \sup_{|\b| \leq k} \, \sup_\xi \, (1 + |\xi|^2)^{(|\b| - m)/2} \| \d_\xi^\b \s(\cdot, \xi)\|_{H^s_\e} < \infty,$$
where $\| \cdot \|_{H^s_\e}$ is defined in \eqref{hse}. 
 Symbols in $S^m_{1,0}$ that do not depend on $x$ are called Fourier multipliers of order $m.$ To a symbol $\s \in \G^m_s,$ we associate the pseudo-differential operator $\op_{\e}(\s)$ defined by its action on ${\cal S}(\R^d)$ as
 \begin{equation} \label{semicl} \op_{\e}(\s) u := (2 \pi)^{-d/2} \int_{\R^{d}} e^{i x \cdot \xi} \sigma(x, \e \xi) \hat u(\xi) \,d\xi.
 \end{equation}
Let 
$$ {\bf M}^m_{k,k'}(\s) := \sup_{|\b| \leq k} \sup_{|\b'| = k'} \, \sup_\zeta \, (1 + |\xi|^2)^{(|\b| - m)/2} | \d_\xi^\b \d_x^{\b'} \s(\cdot, \xi)|_{L^\infty}.
$$
Given a symbol $\s \in \G^m_s,$ if $s  > k' + d/2,$ there holds ${\bf M}^m_{k,k'}(\s) < \infty.$

   The following three lemmas  describe the action, composition, and adjoints of operators with symbols in $\G^m_s,$ based on the results of \cite{MZ0,L}, and the identity 
 \begin{equation} \label{he} \op_\e(\s) u = (h_\e)^{-1} \op_1(\tilde \s) h_\e, \qquad \tilde \s(x,\xi) := \s(\e x, \xi),
 \end{equation}
 relating classical and semiclassical quantizations, where $(h_\e f)(x) := \e^{d/2} f(\e x),$ so that $\| h_\e f\|_{1,s} = \| f \|_{\e,s}.$ In the statements of these results, we shorten ${\bf N}^m_{k,s}$ and ${\bf M}^m_{k,k'}$ into ${\bf N}^m_s$ and ${\bf M}^m_{k'},$ where it is understood that a certain number of derivatives in $\xi,$ depending only on $d,$ are involved in the semi-norms.

\begin{lem} \label{action} Given $m \in \R,$ $s \geq s_0 > d/2,$ and $\s \in \G^m_s,$ for all $u \in H^{s+m},$ there holds 
 $$ \| \op_\e(\sigma) u \|_{H^s_\e} \leq {\bf M}^m_{0}(\sigma) \| u \|_{H^{s+m}_\e} + \e^{- d/2} {\bf N}^m_{s}(\sigma) \| u \|_{H^{s_0 + m}_\e}. $$
\end{lem}

\begin{proof} Use Theorem 1 in \cite{L}, and \eqref{he}.
\end{proof}

 \begin{lem} \label{composition} Let $m_1, m_2, s_2 \in \R,$ and $s_0 > d/2.$ Let $\s_1$ be a Fourier multiplier of order $m_1,$ and $\s_2 \in \G^{m_2}_{s_2}.$ If $s_2 \geq s_0 + \max(m_1,0) + 1,$ there holds
  $$ \op_\e(\s_1) \op_\e(\s_2) - \op_\e(\s_1 \s_2) = \e R(\s_1, \s_2),$$
where for all $s_0 \leq s \leq s_2 - \max(m_1,0),$ for all $u \in H^{s + m_1 + m_2 - 1},$ 
  \begin{equation} \label{compo-pseudo} \begin{aligned}\| R (\s_1,\s_2) u \|_{H^s_\e}  & \lesssim {\bf M}_0^{m_1}(\s_1) {\bf M}^{m_2}_{1}(\s_2) \| u \|_{H^{s + m_1 + m_2 -1}_\e} \\ & \quad + \e^{-1-d/2} {\bf M}_0^{m_1}(\s_1) {\bf N}^{m_2}_{s + \max(m_1,0)}(\s_2) \| u \|_{H^{s_0 + m_1 + m_2 - \max(m_1,0)}_\e},\end{aligned}\end{equation}
\end{lem}

\begin{proof} Use Theorem 3(ii) in \cite{L}, and \eqref{he}.
\end{proof}

 \begin{lem} \label{adjoint} Given $m \in \R,$ $s \geq 1 + s_0 > 1 + d/2,$ and $\s \in \G^m_s,$ there holds for all $u \in H^{s + m - 1},$ 
  $$ \big\| \big( \op_\e(\sigma)^* - \op_\e(\sigma^*) \big) u \big\|_{H^s_\e} \lesssim \e {\bf M}_1(\s) \| u \|_{H^{s + m - 1}_\e} + \e^{ - d/2} {\bf N}_s(\s) \| u \|_{H^{s_0 + m}_\e}.$$
  \end{lem}
  
  \begin{proof} A direct consequence of Lemma \ref{action} and Proposition B.22 in \cite{MZ0}.
  \end{proof}

\begin{section}{Application: small-amplitude shock profiles for
quasilinear relaxation equations with characteristic velocities}\label{sec:app2}

We consider finally the problem of existence of relaxation profiles 
\begin{equation}\label{relaxprof}
U(x,t)=\bar U(x-st), \quad \lim_{z\to \pm \infty}\bar U(z)=U_\pm
\end{equation}
of a relaxation system 
 $$ \d_t U +A(U) \d_x U= Q(U),$$
with
\begin{equation}\label{block}
\quad U= \begin{pmatrix} u\\v \end{pmatrix},
\quad A= \begin{pmatrix} A_{11}& A_{12}\\A_{21}& A_{22}\end{pmatrix},
\quad Q=\begin{pmatrix} 0\\q\end{pmatrix},
\end{equation}
in one spatial dimension, $u\in \RR^n$, $v\in \RR^r$, where, 
for some smooth $v_*$ and $f$, some $\theta > 0,$
\begin{equation}\label{qassum1}
q(u,v_*(u))\equiv 0, \quad  
\Re e\,\sigma (\partial_v q(u,v_*(u)))\le -\theta < 0,
\end{equation}
$\sigma(\cdot )$ denoting spectrum, and
\begin{equation}\label{Aform}
\quad \begin{pmatrix} A_{11}& A_{12}\end{pmatrix}= 
\begin{pmatrix} \partial_u f& \partial_v f\end{pmatrix}.
\end{equation}
Here, we are thinking particularly of the case $n$ bounded
and  $r \gg 1$ arising through discretization or moment closure
approximation of the Boltzmann equation or other kinetic models;
that is, we seek estimates and proof independent of the dimension
of $v$.
Recall, for Boltzmann's equation and its finite approximants,
that $n=5$ is the dimension of the equilibrium ($u$) system
corresponding to standard gas-dynamical flow, whereas the total
dimension $n+r$ may be arbitrarily large: for example, it is 
infinite for the continuous Boltzmann equations and $13$ for the
Grad $13$-moment approximation, with an increasing number of moments as
the desired level of accuracy is increased.

For fixed $n$, $r$, the existence problem was treated in
\cite{YZ, MaZ1} under the additional assumption 
$\det (A-sI)\ne 0$
corresponding to nondegeneracy of the traveling-wave ODE,
using standard center-manifold techniques for amplitudes
$U_+-U_-$ sufficiently small.
However, as pointed out in \cite{MaZ1,MaZ2}, this assumption is
satisfied in general only (by considerations coming from
the subcharacteristic condition) for $2\times 2$ models $r=n=1$,
and is unrealistic for larger models ($n>1$ or $r>1$).
Moreover, it is not satisfied for
the (infinite-dimensional)
 Boltzmann equations, for which the eigenvalues of $A$ are
constant particle speeds of all values, hence cannot be {\it uniformly}
satisfied for discrete velocity or moment closure approximations
as the number of modes goes to infinity, 
at least if they are faithful (consistent) models of Boltzmann.
For, the set of characteristic speeds, given by the eigenvalues
of $A$, in that case must approach
a dense set in the limit as the number of modes goes to infinity,
and so $A$ cannot be uniformly invertible.
Thus, the region of validity for such center manifold arguments
as in \cite{YZ, MaZ1}
in general shrinks
to zero as the number of modes goes to infinity.

A different argument for small-amplitude stability 
based on Chapman--Enskog expansion and Picard iteration was presented
in \cite{MZ1} for the semilinear case $A\equiv \const$.
This yields results independent of dimension; indeed, with
slight modifications, it has been applied to the infinite-dimensional 
Boltzmann equation itself \cite{MZ2}.
However, in the quasilinear case, there seems to be an unavoidable
loss of derivatives in the iteration process, and so
the argument of \cite{MZ1} does not close.  This has 
been remedied in \cite{MTZ} using the Nash--Moser iteration
of the present paper.  We describe this application here
in a simplified case that 
illustrates the main issues while avoiding some
technical details; for the general case, see \cite{MTZ}.

\subsection{Assumptions}\label{assrelax}
Let $f$, $A$, $Q\in C^\infty.$ We assume the following:
\begin{itemize}
\item[(i)] $f$ {\it scalar}, corresponding to $n=1,$ $u \in \R.$

\item[(ii)] $A$ symmetric.

\item[(iii)]  $Q=\begin{pmatrix} 0 & 0\\0 & Q_{22}\end{pmatrix}$
block diagonal, 
with $\Re e \,Q_{22}:=\frac{1}{2}(Q_{22}+Q_{22}^T)$ negative definite and $v_*(u)\equiv 0$.

\item[(iv)] $A_{12}$ nonvanishing.

\item[(v)] $f_*(u) := f(u,0)$ genuinely nonlinear in the sense of Lax, that is $d^2f_*(u)\ne 0$. 

\end{itemize}

 In the general case, (ii) and (iii) can be achieved by 
coordinate transformations \cite{MTZ}. Under (ii) and (iii), condition (iv) is the Kawashima genuine coupling condition, a consequence of which 
is that  the skew matrix 
  $$K:=\bp 0 & A_{12}\\
-A_{21} & 0\ep$$ satisfies
\be\label{kaw}
\Re e\,(KA - Q)\geq c \Id ,
\quad
\ee
for some $c > 0,$ uniformly in $x \in \R.$ Associated with \eqref{block} is a scalar viscous conservation law
\be\label{cons}
\d_t u + \d_x f_*(u)= \d_x (b_*(u) \d_x u), 
\ee
obtained by Chapman--Enskog expansion (described partly below),
with $f_*$ defined in (v) above, and 
\be\label{conscoefs}
b_*(u):=-A_{12}Q_{22}^{-1}A_{21}(u,0).
\ee
By our structural assumptions, 
\be\label{goodb}
\Re e\, b_*\ge \theta>0.
\ee

Taking without loss of generality $s=0$, we study the traveling-wave ODE 
\begin{equation}\label{relax}
A(U)U'= Q(U).
\end{equation}

\subsection{Chapman--Enskog approximation}\label{CEapprox}

Integrating the first equation of \eqref{relax}
, we obtain
\begin{equation}\label{intprof}
\begin{aligned}
f(u,v)&=  f_*(u_\pm),\\
A_{21}(u,v)u'+ A_{22}(u,v)v'&=q(u,v),
\end{aligned}
\end{equation}
where $f_*$ is defined in (v), Section \ref{assrelax}. Taylor expanding the first equation, we obtain
\begin{equation}\label{T1}
f_*(u) + f_v(u,0)v + O(v^2)= f_*(u_\pm).
\end{equation}
Taylor expanding the second equation
and inverting $\partial_v q$,
we obtain
\begin{equation}\label{T2}
v=\partial_v q(u,0)^{-1}A_{21}(u,0) u' + O(|v|^2) +O(|v||u'|) + O(|v'|).
\end{equation}
Substituting \eqref{T2} into \eqref{T1} and rearranging, 
we obtain the approximate viscous profile ODE
\begin{equation}\label{approxprof}
b_*(u)u'= f_*(u) -f_*(u_\pm)  + O(v^2) +O(|v||u'|) + O(|v'|).
\end{equation}

Motivated by \eqref{T2}--\eqref{approxprof}, we define an approximate
solution $(u_{CE}, v_{CE})$ of \eqref{intprof} by choosing 
$u_{CE}$  as a solution of 
\begin{equation}
\label{NS}
b_*(u_{CE})u_{CE}' = f_*(u_{CE}) -f_*(u_\pm),
\end{equation}
and $v_{CE}$  as the first approximation given by \eqref{T2} 
\begin{equation}
\label{NSv1}
\begin{aligned}v_{CE}   =    c_* (u_{CE}) 
 u_{CE}'.
\end{aligned}
\end{equation}
Here, \eqref{NS} can be recognized as the traveling-wave ODE 
associated with approximating scalar viscous conservation law \eqref{cons},
with $s=0$.
From standard scalar ODE considerations (normal forms), we
obtain the following description of solutions. 

\begin{prop}\label{NSprofbds} Under the assumptions of Section {\rm \ref{assrelax},}
for $u_0$ such that $df_*(u_0)=0$,
in a neighborhood of 
$(u_0, u_0)$ in $\RR^1 \times \RR^1$, 
there is a smooth  curve $\cS$ passing through $(u_0, u_0)$,  such that 
for $(u_-, u_+) \in \cS$ with   amplitude $\eps:=|u_+ -u_-| > 0$ 
sufficiently small,
the zero speed shock profile equation   
\eqref{NS} has  a unique (up to translation) 
solution   $u_{CE}$ local to $u_0$.
The shock profile is necessarily of {\rm Lax type}: i.e., with
$df_*(u_-)>0>df_*(u_+)$.

Moreover, 
there is  $\theta>0$ and for all $k$ there is $C_k $ independent of $(u_-, u_+) $ and $\eps$,   
such that 
\begin{equation}\label{NSbds}
|\partial_x^k (u_{CE}-u_\pm)|\le C_k \eps^{k+1}e^{-\theta \eps|x|},
\quad x\gtrless 0. 
\end{equation}
\end{prop}

We denote by 
 $\cS_+$  the set  of $(u_-, u_+) \in \cS $  with  amplitude $\eps:=|u_+ -u_-| > 0$ 
sufficiently small that the profile $u_{CE}$ exists.  
Given $(u_-, u_+) \in \cS_+  $ with associated profile $u_{CE}$, 
we define $v_{CE} $ by \eqref{NSv1} and 
    \begin{equation}
    \label{CE}
    U_{CE} := (u_{CE}, v_{CE}).
    \end{equation}
    It  is an approximate solution of \eqref{intprof} in the following sense: 

\begin{cor}\label{redbds}
For fixed $u_-$ and amplitude $\eps:=|u_+-u_-|$ sufficiently small,
\begin{equation}\label{eq:resbds}
\begin{aligned}
\cR_u&:= f(u_{CE},v_{CE})- f_*(u_\pm)
= O(|u_{CE}'|^2) =O(\eps^4e^{-\theta \eps|x|}),\\
\cR_v&:= g(u_{CE},v_{CE})'-q(u_{CE},v_{CE}) 
= O(|u_{CE}''|)=O(\eps^3 e^{-\theta \eps|x|})
\end{aligned}
\end{equation}
satisfy
\begin{equation}\label{L2resbds}
\begin{aligned}
|   \partial_x^k \cR_u  (x) | 
&\le  C_{k} \eps^{k+4}e^{-\theta \eps|x|} , \\
|   \partial_x^k \cR_v  (x) | 
&\le  C_{k} \eps^{k+3}e^{-\theta \eps|x|} , 
\quad x\gtrless 0,\\
\end{aligned}
\end{equation}
where $C_{k}$   is  independent of $(u_-, u_+) $ and  $\eps=|u_+ -u_-|$. 
\end{cor}

\begin{proof}
For $k=0$,
bounds \eqref{L2resbds} follow by expansions \eqref{T1} and \eqref{T2}, 
definitions \eqref{NS} and \eqref{NSv1}, and bounds \eqref{NSbds}.  
Bounds for $k>0$ follow similarly.
\end{proof}

\begin{rem}
One may continue this process to obtain
Chapman--Enskog approximations $(u_{CE}^N, v_{CE}^N)$
to all orders, with truncation errors 
$( \partial_x^k \cR_u^N  , \partial_x^k \cR_v^N  )
\sim (\eps^{N+k+4},\eps^{N+k+3})$
{\rm \cite{MTZ}.}
\end{rem}

\subsection{Statement of the main theorem}
We are now ready to state the main result.
Define a base state $U_0=(u_0,0)$ with $df_*(u_0)=0,$ and a
neighborhood $\cU
$
of $U_0.$

\begin{theo}\label{main}
Under the assumptions of Section {\rm \ref{assrelax}},
there are $\eps_0 > 0$  and 
$\delta > 0$ such that for $(u_-, u_+) \in \cS_+$ with  amplitude $\eps:=|u_+-u_-| \le \eps_0$,   the standing-wave equation 
\eqref{relax} has a solution   
$\bar U$ in $\cU$, 
 with associated Lax-type 
equilibrium shock $(u_-,u_+)$, satisfying for all $k  $:
\begin{equation}\label{finalbds}
\begin{aligned}
\big|\partial_x^k (\bar U- U_{CE})\big|
&\le C_{k} \eps^{k+2}e^{-\delta  \eps|x|},\\
|\partial_x^k (\bar u-u_\pm)|&\le C_k \eps^{k+1}e^{-\delta \eps|x|},
\quad x\gtrless 0,\\
\big|\partial_x^k (\bar v-v_*(\bar u)\big|
&\le C_k \eps^{k+2}e^{-\delta  \eps|x|},\\
\end{aligned}
\end{equation}  
where $U_{CE}=(u_{CE}, v_{CE})$ is the 
approximating Chapman--Enskog profile defined in \eqref{NS}, and
$C_k$ is independent of  $\eps$. 
Moreover, up to translation, this solution is unique
within a ball of radius $c\eps$ about $U_{CE}$ in norm 
\be \label{norm-th} \eps^{-1/2}\|\cdot\|_{L^2}+\eps^{-3/2}\|\partial_x \cdot\|_{L^2}+\dots+
 \eps^{-11/2}\|\partial_x^5 \cdot\|_{L^2}, \ee for $c>0$ sufficiently small
and $K$ sufficiently large.
\end{theo}
 
That is, behavior of profiles
is well-described by Chapman--Enskog approximation. By (iii), the equilibrium $v_*$ in \eqref{finalbds} is $v_* \equiv 0.$ Note that $U_{CE} - U_\pm$ is order $O(\e)$ in the norm \eqref{norm-th}, by \eqref{finalbds}{\rm (ii)--(iii)}. A consequence of the bounds \eqref{finalbds}, via \cite{MaZ3}, is that the Chapman-Enskog profiles are spectrally stable; see \cite{MTZ}.

\subsection{Functional equation and spaces} \label{pertsec}

Defining the perturbation variable $U:= \bar U- U_{CE}$,
where $U_{CE}$ is defined in \eqref{CE},
we obtain from \eqref{intprof} the nonlinear perturbation equations
$\Phi^\eps(U)=0$, where
\begin{equation}\label{intpert}
\Phi^\eps(U):=
\begin{pmatrix}
 f(U_{CE}+U)-f_*(u_-) \\
A_{21}(U_{CE} +U) (u_{CE} +u)'
+A_{22}(U_{CE} +U) (v_{CE} +v)'
-q(U_{CE}+U)
\end{pmatrix}.
\end{equation}
Formally linearizing $\Phi^\eps$ about a background profile $\un U$,
we obtain 
\begin{equation}\label{Phi'}
(\Phi^\eps)'(\un U)U=
\begin{pmatrix}
A_{11} u +A_{12} v \\
A_{21} u'+  A_{22} v' + b_2 U
- \d_v q \,v 
\end{pmatrix},
\end{equation}
where 
$$A= A(U_{CE}+\un U), \quad \d_v q = \d_v q (U_{CE}+\un U),$$ and $$b_2 U = \big( \d_u (A_{21} + A_{22})(U_{CE}+\un U) \cdot u + \d_v (A_{21} + A_{22})(U_{CE}+\un U) \cdot v\big) 
(U_{CE}+\un U)'.$$

The associated linearized equation for a given forcing term
$h = (h_1, h_2)$ is 
\begin{equation}\label{linpert}
(\Phi^\eps)'(\un U) U= h.
\end{equation}

The coefficients and the error term $\cR$ from Corollary \ref{redbds} are smooth functions of 
$U_{CE}$ and its derivatives, so behave like smooth functions of 
$ \eps x$. Thus, it is natural to solve the equations in spaces which reflect 
this scaling. We observe that 
 \begin{equation} \label{scaling} \| f (\e \cdot) \|_{L^2} = \e^{-1/2} \| f\|_{L^2}, \quad \| f(\e \cdot)\|_{H^s} = \e^{-1/2} \sum_{k=0}^s \e^k \| \d_x^k f \|_{L^2},\end{equation}
 in one space dimension, for $s \in \NN.$ We do not introduce explicitly the change of variables
$\tilde x = \eps x$, but introduce 
exponentially weighted norms which correspond to usual weighted $H^s$ norms 
in the $\tilde x $ variable: for $s \in \NN$ and $\delta \geq 0,$ we let, in accordance with \eqref{scaling},
\begin{equation} \label{defnorm}
\|f \|_{\e,\delta,s} :=  \eps^{1/2} \sum_{0 \leq k \le s} \eps^{-k}  \|e^{\delta \eps (1 + |\cdot|^2)^{1/2}} \d_x^k f \|_{L^2},
\end{equation}
the exponential weight accounting for the exponential decay of the source and the solution. For fixed $\delta$, we introduce the spaces $E_s:=H^s(\RR),$
and
$F_s:= H^{s+1}(\RR) \times H^s(\RR),$ with norms
 $$| h |_{E_s} := \| h \|_{\e,\delta,s}, \qquad |(h_1,h_2)|_{F_s} := \|h_1 \|_{\e,\delta,s+1} + \| h_2 \|_{\e,\delta,s}.$$
 In particular, the Chapman-Enskog approximate solution of Section \ref{CEapprox} satisfies, by \eqref{NSbds},
 \begin{equation} \label{bd-CE}
  |\d_x^j U_{CE}|_{L^\infty} \leq \e^{j + 1} C_{j}, \quad  | \d_x^{j+1} U_{CE} |_{E_s} \leq \e^{j + 2} C_{j,s}, \qquad \mbox{for $j \geq 0,$}
  \end{equation}
 where the constants $C_{j} > 0,$ $C_{j,s} > 0$ do not depend on $\e,$ for all $s \in \NN.$

\begin{rem} \label{rem:Moser-app2} Moser's inequality in the weighted norms \eqref{defnorm} is
 $$ \| f g \|_{\e,\delta,s} \lesssim |f|_{L^\infty} \| g \|_{\e,\delta,s} +  \| f \|_{\e,\delta,s} |g|_{L^\infty},  \quad s \geq 0, \, f,g \in L^\infty \cap H^s,$$
 and the Sobolev embedding has norm
 $$ | \d_x^k f |_{L^\infty} \lesssim \e^{-1/2} \| f \|_{\e,\delta, k + 1 + [d/2]}, \qquad k \geq 0, \, f \in H^{k  +1 + [d/2]}.$$
\end{rem}

\subsection{Nash-Moser iteration scheme}

\begin{lem} \label{lem1} The application $\Phi^\e,$ defined in \eqref{intpert}, maps smoothly $E_s$ to $F_{s-1},$ for any $s.$ It satisfies Assumption {\rm \ref{ass1}} with $s_0 = 1,$ $\g_0 = \g_1 = 1/2,$ $s_1 = + \infty,$ and Assumption {\rm \ref{ass-wkb}}, with $k = N.$
\end{lem}

\begin{proof} The bounds of Assumption \ref{ass1}, describing the action of $\Phi^\e$ and its first two derivatives, follow directly from Moser's inequality and the definition of the weighted Sobolev norms. 
The bound on $\Phi^\e(0)$ is immediate from \eqref{L2resbds} and \eqref{defnorm}.
\end{proof}

\begin{prop} \label{invprop} Under the assumptions of Theorem {\rm \ref{main},} for $\e$ and $\delta$ small enough, the map $\Phi^\e$ satisfies Assumption {\rm \ref{ass2}} with $r = 1,$ $r' = 0,$ $\g = 1,$ and $\k = 1.$ 
\end{prop}

The proof of this proposition is carried out in 
Sections \ref{sec:linest}.
Once it is established, existence and uniqueness follow by Theorems \ref{th1} and \ref{th2}:

\begin{proof}[Proof of Theorem {\rm \ref{main}} (Existence)]
The profile $U_{CE}$ exists if $\eps$ is small enough. 
Comparing, we find that
Lemma \ref{lem1},  Proposition \ref{invprop},
and Corollary \ref{redbds} verify, respectively, Assumptions \ref{ass1},
\ref{ass2}, and \ref{ass-wkb} of our Nash--Moser iteration scheme,
with $s_0=3$, $\gamma_0= \g_1 = 1/2$, $\g=1$, $m=r=1$, $r'=0$, arbitrary $s_1,$ and $k = N$ large enough.
Taking $s_1$ sufficiently large, and applying
Theorem \ref{th1}, we thus
obtain existence of a solution $U^\eps$ of \eqref{intpert}
with $|U^\eps|_{H^{s+1}_{\eps, \delta}}\le C\eps^{2}$.
Defining $\bar U^\eps:=U_{CE}^\eps+U^\eps$, and
noting by Sobelev embedding that $|h|_{H^{s+1}_{\eps, \delta}}$
controls $| e^{\delta \eps |x|}h|_{L^\infty}$, we obtain the result.
\end{proof}

\begin{proof}[Proof of Theorem {\rm \ref{main}} (Uniqueness)]
Applying Theorem \ref{th2}
for $s_0=3$, $\gamma_0=\g_1 = 1/2$, $\g=1$, $k=3$, $m=r=1$, $r'=0$, 
we obtain uniqueness in a ball of radius $c\eps$ in 
$H^{4}_{\eps,0}$, 
$c>0$ sufficiently small,
under the additional phase condition \eqref{phasecond}.
We obtain unconditional uniqueness from this weaker version
by the observation that phase condition
\eqref{phasecond} may be achieved for any solution $\bar U=U_{CE}+U$
with 
$$
\|U'\|_{L^\infty}\le c \eps^{2}<< U_{CE}'(0)\sim \eps^2
$$
by translation in $x$, yielding
$\bar U_a(x):=\bar U(x+a)= U_{CE}(x)+ U_a(x)$
with 
$$
U_a(x):= U_{CE}(x+ a)-U_{CE}(x)+ U(x+a) 
$$
so that, defining $\phi:=\bar U'/|\bar U'|$, 
we have
$\partial_a \langle \phi, U_a\rangle \sim 
\langle \phi, U_{CE}' + U' \rangle
=\langle \phi, (1+o(1))\bar U' + U' \rangle
= (1+o(1)) |\bar U'|\sim \eps^2$
and so (by the Implicit Function Theorem applied to $h(a):=\eps^{-2}
\langle \phi, U_a\rangle$, together with the
fact that $\langle \phi, U_0\rangle = o(\eps)$ and
that $\langle \phi, \bar U_{NS}'\rangle\sim |\bar U_{NS}'|\sim \eps^{2}$)
the inner product $\langle \phi, U_a\rangle$, hence also
$\Pi U_a$
may be set to zero by appropriate choice of $a=o(\eps^{-1})$ leaving
$U_a$ in the same $o(\eps)$ neighborhood, by the computation
$U_a-U_0\sim \partial_a U \cdot a\sim o(\eps^{-1})\eps^2$
\end{proof}

\subsection{Linearized estimates} \label{sec:linest}

We here carry out the main step in the proof
 of obtaining corresponding A Priori estimates; see \ref{prop73} below.
The remaining step of demonstrating existence for the linearized
problem can be carried out
by the vanishing viscosity method as in \cite{MZ2}, 
with viscosity coefficient $\eta>0$, obtaining
existence for each positive $\eta$ by standard boundary-value theory,
and noting that the A Priori bounds 
\eqref{invbdHs7} of Proposition \ref{prop73}
 persist under regularization for
sufficiently small viscosity $\eta>0$, so that we can obtain
a weak solution in the limit by extracting a weakly convergent
subsequence.
We omit this step, referring the reader to
Section 8, \cite{MZ1}, for details.
The asserted estimates then follow in the limit by continuity.

The rest of this subsection is devoted to establishing the asserted
A Priori estimates.

\subsubsection{Internal and high frequency estimates}
 \label{energy}

  Let $s \in \NN,$ and some background profile $\un U \in H^s.$
 We consider equation \eqref{linpert}, and its differentiated form: 
\begin{equation}\label{apriorieq}
(AU'- dQ+b)U= (h_1',h_2),
\end{equation}
in which $b U := (b_1 U, b_2 U),$ where $b_2$ is defined in Section \ref{pertsec}, and $b_1$ is defined similarly, by differentiating the coefficients $A_{11},$ $A_{12}$ in the first line of \eqref{linpert}. The coefficients $A,$ $b,$ and $dQ$
are smooth functions of $U_{CE}+\un U.$ The bound for $U_{CE},$ \eqref{bd-CE}, and the assumed bound for $\un U$ imply the coefficient bounds 
\be \label{coeffests} \left\{ \begin{aligned}
|\partial_x^{j+1} C|_{L^\infty} + |\partial_x^j b|_{L^\infty} & \le c_j \eps^{2+j}, & 0 \leq j \leq s-1,\\
\|\partial_x^{k+1} C\|_{L^2} + \|\partial_x^{k} b\|_{L^2}
& \le C_k \e^{1/2 + k} (\e + |\un U|_{\e,0,s+1}), & 0 \le k \le s,\end{aligned} \right.\ee
where $C = A, Q, K,$ the matrix $K$ being the Kawashima multiplier introduced in Section \ref{assrelax}. In \eqref{coeffests}, the constants $c_j$ depend on $| \d_x^{j'} (U_{CE} + \un U)|_{L^\infty},$ for $0 \leq j' \leq j,$ while, by the classical Moser's inequality, the constants $C_k$ depend on $|U_{CE} + \un U|_{L^\infty}.$

We give in the following Proposition an estimate for the internal variables $U' = (u', v')$ and $v$. 

\begin{prop}\label{estHs} For 
$k \ge 1$, for come $C > 0,$ for $\e$ and $\delta$ small enough, given $ h \in F_{k+1},$ if $U \in H^k$ satisfies \eqref{apriorieq} with  $|\un U|_{E_2}\le \eps$, there holds
\begin{equation}\label{sharph2eq}
\begin{aligned}
  | \D_x^k U' |_{E_0} +   | \partial^k_x v |_{E_0}  
 &  \le  C  \big(
  | \partial_x^k H |_{E_0}  +   \eps^k  \big( 
 | U' |_{E_{k-1}} + \eps   | v |_{E_{k-1}}  +
  \eps | u |_{E_0}\big)\big)  \\
&\quad +
C \eps^{k+1} |\un U|_{E_{k+2}}
(  |v |_{E_1} +
\eps|U |_{E_2}),
\end{aligned}
\end{equation}
where $H = (h_1, h'_1, h''_1, h_2, h'_2).$
\end{prop}

In order to prove Proposition \ref{estHs}, we start with an $L^2$ estimate for the internal variables:

\begin{lem}
\label{lem62}
For some $C > 0,$ for $\eps$ sufficiently small, 
given $(h_1, h_2) \in H^2 \times H^1,$ if $U \in H^1$ satisfies \eqref{apriorieq} with $\|\un U\|_{\e,0,2}\le \eps$,
there holds
\begin{equation}\label{sharph1eq}
\| U'  \|_{L^2} + \| v \|_{L^2}  \leq C    
\big(\|  h_1 \|_{H^2} + \| h_2 \|_{H^1}  
+ \e \| u \|_{L^2}\big). 
\end{equation}
\end{lem}

\begin{proof}[Sketch of proof] The key is to bound the the $L^2$ scalar product $({\mathfrak S} h, U)_{L^2}$ from above and from below, where
 ${\mathfrak S}$ is the symmetrizer
$
{\mathfrak S} = \d_x^2    + \d_x \circ K   -  \lambda,$
for an appropriate choice of $\lambda \in \RR,$ using symmetry of $A,$ and positivity of $K A - Q$ \eqref{kaw}. A complete proof in given in Section 5.4.1, \cite{MTZ}.
\end{proof}

\begin{proof}[Proof of Proposition~\ref{estHs}] We use Lemma \ref{lem62} for $\e^{1/2} e^{\delta \e \langle x \rangle} U,$ which solves \eqref{apriorieq} with the source term
 $$ \e^{1/2} e^{\delta \e \langle x \rangle} \big( (h'_1,h_2) +  \delta \e \langle x \rangle' \tilde A U\big).$$
 This gives
   \begin{equation}\label{invbd}
 | U' |_{E_0}   + | v |_{E_0}  \le 
C  \big( | H |_{E_0} +  \e  | u |_{E_0} \big),
\end{equation}
i.e., estimate \eqref{sharph2eq} with $k = 0.$ Estimate \eqref{sharph2eq} with $k > 0$ is obtained in a simialar way, differentiating \eqref{apriorieq} $k$ times. For more details, see Proposition 5.5, \cite{MTZ}.
\end{proof}

\subsubsection{Linearized Chapman--Enskog estimate} \label{linCEestimates}

 It remains only to estimate the weighted $L^2$ norm $|u|_{E_0}$ in order to close the estimates
and establish the bound claimed in Proposition \eqref{invprop}.
To this end, we work with the first equation  in 
\eqref{linpert}
and  estimate it by comparison with the Chapman-Enskog 
approximation of Section~\ref{CEapprox}. 
 From the second equation in \eqref{linpert}, in which, by \eqref{coeffests}, $b = O(\e^2),$   
we find, for small $\e,$ 
\begin{equation}
\label{T2s7}
v= \left(\partial_v q - b_{22}\right)^{-1}  
\Big(A_{21} u' + A_{22} v' + b_{21} u - h_2 \Big),
\end{equation}
where $b_2 U =: b_{21} u + b_{22} v.$ 
Introducing now \eqref{T2s7} in the first equation of \eqref{linpert}, we obtain the linearized profile equation
\begin{equation} \label{intpertu}
 A_{12} (\d_v q - b_{22})^{-1} A_{21} u' + \big(A_{11} + A_{12} (\d_v q - b_{22})^{-1} b_{21} \big) u = h^\sharp,
 \end{equation}
where $h^\sharp$ depends on the source $h$ and on $v',$ but not on $v$ nor on $u:$
 $$ h^\sharp := - A_{12} (\d_v q - b_{22})^{-1} A_{22} v' + h_1 + A_{12} (\d_v q - b_{22})^{-1} h_2.$$
 Introduce the notation
  $$ \begin{aligned} b^\sharp & := \left(A_{12} (\d_v q - b_{22})^{-1} A_{21}\right)(U_{CE} + \cdot), \\ f^\sharp & := \left(A_{11} + A_{12} (\d_v q - b_{22})^{-1} b_{21}\right)(U_{CE} + \cdot).\end{aligned}$$
  Then \eqref{intpertu} takes the form
  \begin{equation} \label{intpertu2}
   (b^\sharp \d_x - f^\sharp)(\un U) u = - h^\sharp.
  \end{equation}
  We estimate the solution of \eqref{intpertu2} by the following:

\begin{prop}\label{uprop}
Given $\un U \in H^4,$ with $| \un U|_{E_4}\le \eps$, if $\e$ is sufficiently small,  then
the operator $( b^\sharp \partial_x - f^\sharp)(\un U)$ has
a right inverse $( b^\sharp \partial_x - f^\sharp)(\un U)^{\dagger},$ satisfying the bound  
\begin{equation}\label{rightinv}
\|( b^\sharp \partial_x - f^\sharp)(\un U)^{\dagger}h\|_{E_0} \le 
C\eps^{-1}\|h\|_{E_0},
\end{equation}
and uniquely specified by the property that the solution 
$u$ to \eqref{intpertu2}  satisfies 
\begin{equation}\label{phase}  
\ell_\eps  \cdot u(0) =0. 
\end{equation}
for certain unit vector $\ell_\eps$. 
\end{prop}

\begin{proof}
Working in $\tilde x=\eps x$ coordinates, and noting that
$\eps^{-1} |f^\sharp(0)- f^\sharp(U_\pm)|\sim e^{-\theta|\tilde x|}$,
by \eqref{NSbds}, we obtain using $\partial_{x}=\eps \partial_{\tilde x}$
the equation 
\be\label{modeq}
(b^\sharp\partial_{\tilde x} -\eps^{-1}f^\sharp)u= \eps^{-1} h,
\qquad u(0)=0.
\ee
This is a rather standard boundary-value
ODE problem with exponentially convergent coefficients at spatial infinity.
Using the extra condition $u(0)=0$, we may break it into a pair of
boundary values problems on $(-\infty,0]$ and $[0,+\infty)$, each
of which, by the Lax condition $df_*(u_-)>0>df_*(u_+)$, implying
that there is a one-dimensional manifold of decaying solutions as
$\tilde x\to -\infty$ or as $\tilde x\to +\infty$, is well-posed,
from $H^s_{\eps,\delta}$ to itself, so long as $\delta$ is strictly
smaller that $\eps^{-1}\min |df_*(u_\pm)|$.
Taking account of the $\eps^{-1}$ factor in the righthand side of
\eqref{modeq}, we obtain the result.
\end{proof}

 Combining Proposition \ref{estHs} with $k=1$ and Proposition \ref{uprop}, we obtain:
 \begin{prop}
 \label{prop72}
For some $C> 0,$ for $\e$ and $\delta$ small enough, given $h \in F_2,$ and $\un U \in H^4$ satisfying $| \un U|_{E_4} \le \e,$ if $U = (u,v) \in H^2$ satisfies \eqref{linpert}, with $u$ satisfying \eqref{phase}, there holds  
  \begin{equation}
  \label{invbdH2s7}
| U |_{E_2} \leq C \eps^{-1} |h|_{F_2}.
\end{equation}
\end{prop}

Knowing a bound  for $\| u \|_{L^2_{\eps, \delta}}$, Proposition~\ref{estHs} implies by induction the following final result.

\begin{prop}\label{prop73} For $s \geq 3,$ for some $C >0,$ for $\e$ and $\delta$ small enough, 
given $h \in F_s$ and $\un U \in H^{s+1}$ with $|\un U|_{E_4} \leq \e,$ if $U \in H^s$ satisfies  
  \eqref{linpert} and \eqref{phase},
then
 \begin{equation}\label{invbdHs7}
 |U|_{E_s} \leq
 \eps^{-1} C\big( 
 | \un U |_{E_{s+1}}
| h |_{F_2}
+
| h |_{F_s}\big)
\end{equation}

\end{prop}

Propositon \ref{prop73} can be used to establish Proposition \ref{invprop}
by a vanishing viscosity argument; see \cite{MZ1}.

\begin{subsection}{Why Nash--Moser?}
We conclude by discussing
why we seem to need Nash--Moser to close the argument.
Recall the standard proof of existence for 
quasilinear symmetric hypertolic systems $u_t + A(u)u_x=S$
using energy estimates.  One writes an iteration scheme
$$
u^{n+1}_t+ A(u^n)u^{n+1}_x=S,
$$
which gives $H^s$ bounds 
$|u^{n+1}|_{H^s}\le C|g|_{H^s}$ so long as $|u^{n}|_{H^s}$
is small, and contraction in lower norms on small time intervals,
giving the result.

But, it is easily checked that this does not work for equations
in conservative form
$u_t + (A(u)u)_x=S$, for which
$$
u^{n+1}_t + (A(u^{n})u^{n+1})_x=S,
$$
gives $H^s$ bounds 
$|u^{n+1}|_{H^s}\le C|S|_{H^s}$ rather for $|u^{n}|_{H^{s+1}}$
small, hence involves loss of derivatives.

Usually, for a conservative equation $u_t+f(u)_x=S$, this is no problem,
since we are free to write it in nonconservative form
$u_t+ df(u)u_x=S$.
In the present case, however, it is essential for the key
Chapman--Enskog estimation of the macroscopic variable $u$
that we write the
first row of our equation in integrated form
$f(u,v)=s$,
enforcing a linearization $A_{11}u+A_{12}v= \tilde s$.
But, in the part of our argument in which we control microscopic
variables by energy estimates, we differentiate this equation
and group it with the second row, thus leading
to a partially {\it conservative} form in which the energy estimates
lose a derivative.

That is, the Chapman--Enskog part of our argument does not seem to
be compatible with the nonconservative form needed to close energy
estimates without losing a derivative.  We have not been able to find
a direct way around this (using some alternative scheme), and so for
the moment Nash--Moser iteration appears essential for the argument.

\end{subsection}

\end{section}

\end{document}